%% file: C.tex
\documentclass{birkbook}
\usepackage{amsmath,amssymb,amscd,amsthm,amsxtra}
\usepackage{latexsym}
\usepackage[all]{xy}

\input{formatting}
\input{macros}

\begin{document}

\frontmatter

\makeatletter
\renewcommand{\@title}{Curves and Jacobians over function fields}
\renewcommand{\@author}{Douglas Ulmer\\School of Mathematics\\Georgia
  Institute of Technology\\Atlanta, GA~~30332}
\renewcommand{\@date}{}
\makeatother

\maketitle

\tableofcontents
\newpage
\section*{Preface}
\addcontentsline{toc}{section}{Preface}
\markboth{Preface}{Preface}\thispagestyle{empty} 
These notes originated in a 12-hour course of lectures given at
the Centre de Recerca Mathem\`atica near Barcelona in February, 2010.
The aim of the course was to explain results on curves and their
Jacobians over function fields, with emphasis on the group of rational
points of the Jacobian, and to explain various constructions of
Jacobians with large Mordell-Weil rank.

More so than the lectures, these notes emphasize foundational results
on the arithmetic of curves and Jacobians over function fields, most
importantly the breakthrough works of Tate, Artin, and Milne on the
conjectures of Tate, Artin-Tate, and Birch and Swinnerton-Dyer.  We
also discuss more recent results such as those of Kato and Trihan.
Constructions leading to high ranks are only briefly reviewed, because
they are discussed in detail in other recent and forthcoming
publications.

These notes may be viewed as a continuation of my Park City notes
\cite{Ulmer11}.  In those notes, the focus was on elliptic curves and
finite constant fields, whereas here we discuss curves of high genera
and results over more general base fields.

It is a pleasure to thank the organizers of the CRM Research Program,
especially Francesc Bars, for their efficient organization of an
exciting meeting, the audience for their interest and questions, the
National Science Foundation for funding to support the travel of
junior participants (grant DMS 0968709), and the editors for their
patience in the face of many delays.  It is also a pleasure to thank
Marc Hindry, Nick Katz, and Dino Lorenzini for their help.
Finally, thanks to Timo Keller and Ren\'e Pannekoek for corrections.

I welcome all comments, and I plan to maintain a list of corrections
and supplements on my web page.  Please check there for updates if you
find the material in these notes useful.

\bigskip

{\hfill Atlanta, March 24, 2012}

\mainmatter

\chapter{Dramatis Personae}
\section{Notation and standing hypotheses}
Unless explicitly stated otherwise, all schemes are assumed to be
Noetherian and separated and all morphisms of schemes are assumed to
be separated and of finite type.

A {\it curve\/} over a field $F$ is a scheme over $F$ that is reduced
and purely of dimension 1, and a {\it surface\/} is similarly a scheme
over $F$ which is reduced and purely of dimension 2.  Usually our
curves and surfaces will be subject to further hypotheses, like
irreducibility, projectivity, or smoothness.

We recall that a scheme $Z$ is {\it regular\/} if each of its local
rings is regular.  This means that for each point $z\in Z$, with local
ring $\OO_{Z,z}$, maximal ideal $\m_z\subset\OO_{Z,z}$, and residue
field $\kappa(z)=\OO_{Z,z}/\m_z$, we have
$$\dim_{\kappa(z)}\m_z/\m_z^2 = \dim_zZ.$$
Equivalently, $\m_z$ should be generated by $\dim_zZ$ elements.

A morphism $f:Z\to S$ is {\it smooth\/} (of relative dimension $n$) at
$z\in Z$ if there exist affine open neighborhoods $U$ of $z$ and $V$
of $f(z)$ such that $f(U)\subset V$ and a diagram
\begin{equation*}
\xymatrix{
U\ar[d]_f\ar[r]&\ar[d]{\spec R[x_1,\dots,x_{n+k}]/(f_1,\dots,f_k)}\\
V\ar[r]&\spec R}
\end{equation*}
where the horizontal arrows are open immersions and where the Jacobian
matrix $(\partial f_i/\partial x_j(z))_{ij}$ has rank $k$.  Also, $f$
is {\it smooth\/} if it is smooth at each $z\in Z$.
 
Smoothness is preserved under arbitrary change of base.  If
$f:Z\to\spec F$ with $F$ a field, then $f$ smooth at $z\in Z$ implies
that $z$ is a regular point of $Z$.  The converse holds if $F$ is
perfect, but fails in general.  See
Section~\ref{s:smoothness-regularity} below for one example.  A lucid
discussion of smoothness, regularity, and the relations between them
may be found in \cite[Ch.~V]{MumfordOdaAG2}.

If $Y$ is a scheme over a field $F$, the notation $\overline{Y}$ will
mean $Y\times_F\overline{F}$ where $\overline{F}$ is an algebraic
closure of $F$.  Letting $F^{sep}\subset\overline{F}$ be a separable
closure of $F$, we will occasionally want the Galois group
$\gal(F^{sep}/F)$ to act on objects associated to $\overline{Y}$.  To
do this, we note that the action of Galois on $F^{sep}$ extends
uniquely to an action on $\overline{F}$.  We also note that the
\'etale cohomology groups $H^i(Y\times_F\overline{F},\Ql)$ and
$H^i(Y\times_FF^{sep},\Ql)$ are canonically isomorphic
\cite[VI.2.6]{MilneEC}.

\section{Base fields}
Throughout, $k$ is a field, $\CC$ is a smooth, projective, absolutely
irreducible curve over $k$, and $K=k(\CC)$ is the field of rational
functions on $\CC$.  Thus $K$ is a regular extension of $k$ (i.e.,
$K/k$ is separable and $k$ is algebraically closed in $K$) and the
transcendence degree of $K$ over $k$ is 1.  The function fields of our
title are of the form $K=k(\CC)$.


We view the base field $K$ as more-or-less fixed, except that we are
willing to make a finite extension of $k$ if it will simplify matters.
At a few places below (notably in Sections~\ref{s:XX} and
\ref{s:smoothness-regularity}), we extend $k$ for convenience, to
ensure for example that we have a rational point or smoothness.

We now introduce standing notations related to the field $K$.  Places
(equivalence classes of valuations, closed points of $\CC$) will be
denoted $v$ and for each place $v$ we write $K_v$, $\OO_v$, $\m_v$,
and $k_v$ for the completion of $K$, its ring of integers, its maximal
ideal, and its residue field respectively.

Fix once and for all separable closures $\Ksep$ of $K$ and $\Ksep_v$
of each $K_v$ and embeddings $\Ksep\into\Ksep_v$.  Let $\ksep$ be the
algebraic closure of $k$ in $\Ksep$; it is a separable closure of $k$
and the 
embedding $\Ksep\into\Ksep_v$ identifies 
$\ksep$ with the residue field of $\Ksep_v$.

We write $G_K$ for $\gal(\Ksep/K)$ and similarly for $G_{K_v}$ and
$G_k$.  The embeddings $\Ksep\into\Ksep_v$ identify $G_{K_v}$ with a
decomposition group of $K$ at $v$.  We also write $D_v\subset G_K$ for
this subgroup, and $I_v\subset D_v$ for the corresponding inertia group.

\section{The curve $X$}
Throughout, $X$ will be a curve over $K$ which is always assumed to be
smooth, projective, and absolutely irreducible.  Thus $K(X)$ is a
regular extension of $K$ of transcendence degree 1.

The genus of $X$ will be denoted $g_X$, and since we are mostly
interested in the arithmetic of the Jacobian of $X$, we always assume
that $g_X>0$.

We do not assume that $X$ has a $K$-rational point.  More
quantitatively, we let $\delta$ denote the {\it index\/} of $X$, i.e.,
the $\gcd$ of the degrees of the residue field extensions
$\kappa(x)/K$ as $x$ runs over all closed points of $X$.
Equivalently, $\delta$ is the smallest positive degree of a
$K$-rational divisor on $X$. We write $\delta'$ for the {\it period\/}
of $X$, i.e., the smallest degree of a $K$-rational divisor class.
(In fancier language, $\delta'$ is the order of $\Pic^1$ as an element
of the Weil-Ch\^atelet group of $J_X$.  It is clear that $\delta'$
divides $\delta$.  It is also easy to see (by considering the divisor
of a $K$-rational regular 1-form) that $\delta|(2g_X-2)$.  Lichtenbaum
\cite[Thm.~8]{Lichtenbaum69} showed that $\delta|2\delta^{\prime2}$,
and if $(2g_x-2)/\delta$ is even, then $\delta|\delta^{'2}$.

Similarly, for a closed point $v$ of $\CC$, we write $\delta_v$ and
$\delta'_v$ for the index and period of $X\times_KK_v$.  It is known
that 
$\delta'_v|(g_X-1)$, 
the ratio $\delta/\delta'$ is either 1 or 2,
and it is 1 if $(g_X-1)/\delta'$ is even.  (These facts follow from
the arguments in \cite[Thm.~7]{Lichtenbaum69} together with the
duality results in \cite{MilneADT}.)

\section{The surface $\XX$}\label{s:XX-intro}
Given $X$, there is a unique surface $\XX$ equipped with a morphism
$\pi:\XX\to\CC$ with the following properties: $\XX$ is irreducible
and regular, $\pi$ is proper and relatively minimal (defined
below), and the generic fiber of $\XX\to\CC$ is $X\to\spec K$.

Provided that we are willing to replace $k$ with a finite extension, we
may also insist that $\XX\to\spec k$ be smooth.  We generally make
this extension, and also if necessary extend $k$ so that $\XX$ has a
$k$-rational point.

The construction of $\XX$ starting from $X$ and discussion of its
properties (smoothness, cohomological flatness, ...) will be given in
Chapter~\ref{ch:XX} below.

We note that any smooth projective surface $\SS$ over $k$ is closely
related to a surface $\XX$ of our type.  Indeed, after possibly
blowing up a finite number of points, $\SS$ admits a morphism to
$\P^1$ whose generic fiber is a curve of our type (except that it might
have genus 0).  Thus an alternative point of view would be to start
with the surface $\XX$ and construct the curve $X$.  We prefer to
start with $X$ because specifying a curve over a field can be done
succinctly by specifying its field of functions.

\section{The Jacobian $J_X$}
We write $J_X$ for the Jacobian of $X$, a $g$-dimensional, principally
polarized abelian variety over $K$.  It represents the relative Picard
functor $\Picf^0_{X/K}$.  The definition of this functor and results
on its representability  are reviewed in Chapter~\ref{ch:J}.

We denote by $(B,\tau)$ the $K/k$-trace of $J_X$.  By definition, this
is the final object in the category of pairs $(A,\sigma)$ where $A$ is
an abelian variety over $k$ and $\sigma:A\times_kK\to J_X$ is a
$K$-morphism of abelian varieties.  We refer to \cite{Conrad06} for a
very complete account of the $K/k$-trace in the language of schemes.
We will calculate $(B,\tau)$ (completely in a special case, and
up to inseparable isogeny in general) in Section~\ref{s:S-T} .

One of the main aims of these notes is to discuss the arithmetic of
$J_X$, especially the Mordell-Weil group $J_X(K)/\tau B(k)$ and the
Tate-Shafarevich group $\sha(J_X/K)$ as well as their connections with
analogous invariants of $\XX$.

\section{The N\'eron model $\JJ_X$}
We denote the N\'eron model of $J_X$ over $\CC$ by $\JJ_X\to\CC$ so
that $\JJ_X\to\CC$ is a smooth group scheme satisfying a universal
property.  More precisely, for every place $v$ of $K$, every
$K_v$-valued point of $J_X$ should extend uniquely to a section of
$\JJ_X\times_\CC\spec\O_v\to\spec\O_v$.  We refer to \cite{Artin86a}
for a brief overview and \cite{BoschLutkebohmertRaynaudNM} for a
thorough treatment of N\'eron models.

There are many interesting results to discuss about $\JJ_X$, including
a fine study of the component groups of its reduction at places of
$\CC$, monodromy results, etc.  Due to constraints of time and space,
we will not be able to discuss any of these, and so we will have
nothing more to say about $\JJ_X$ in these notes.

\section{Our plan}
Our goal is to discuss the connections between the objects $X$, $\XX$,
and $J_X$.  Specifically, we will study the arithmetic of $J_X$ (its
rational points, Tate-Shafarevich group, $BSD$ conjecture), the
arithmetic of $\XX$ (its N\'eron-Severi group, Brauer group, Tate and
Artin-Tate conjectures), and connections between them.

In Chapters~\ref{ch:XX} and \ref{ch:J} we discuss the construction and
first properties of $\XX\to\CC$ and $J_X$.  In the following three
chapters, we work out the connections between the arithmetic of these
objects, mostly in the case when $k$ is finite, with the $BSD$ and Tate
conjectures as touchstones.  In Chapter~\ref{ch:complements} we give a
few complements on related situations and other ground fields.  In
Chapter~\ref{ch:cases} we review known cases of the Tate
conjecture and their consequences for Jacobians, and in
Chapter~\ref{ch:ranks} we review how one may use these results to
produce Jacobians with large Mordell-Weil groups.

\chapter{Geometry of $\XX\to\CC$}\label{ch:XX}

\section{Construction of $\XX$}\label{s:XX}

Recall that we have fixed a smooth, projective, absolutely irreducible
curve $X$ over $K=k(\CC)$ of genus $g_X>0$.

\begin{prop}
  There exists a regular, irreducible surface $\XX$ over $k$ equipped
  with a morphism $\pi:\XX\to\CC$ which is projective, relatively
  minimal, and with generic fiber $X\to\spec K$.  The pair $(\XX,\pi)$
  is uniquely determined by these properties.  We have that $\XX$ is
  absolutely irreducible and projective over $k$ and that $\pi$ is
  flat, $\pi_*\OO_\XX=\OO_\CC$, and $\pi$ has connected fibers.
\end{prop}

\begin{proof}
  To show the existence of $\XX$, we argue as follows.  Choose a
  non-empty affine open subset $U\subset X$ and an affine model for
  $U$:
$$U=\spec K[x_1,\dots,x_m]/(f_1,\dots,f_n).$$
Let $\CC^0\subset\CC$ be a non-empty affine open where all of the
coefficients of the $f_i$ (which are elements of $K$) are regular
functions.  Let $R=\OO_\CC(\CC^0)$, a $k$-algebra and Dedekind domain.
Let $\UU$ be the affine $k$-scheme
$$\UU=\spec R[x_1,\dots,x_m]/(f_1,\dots,f_n).$$
Then $\UU$ is reduced and irreducible and the inclusion 
$$R\to R[x_1,\dots,x_m]/(f_1,\dots,f_n)$$ 
induces a morphism $\UU\to\CC^0$ whose generic fiber is $U\to\spec K$.
Imbed $\UU$ in some affine space $\A^N_k$ and let $\XX_0$ be the
closure of $\UU$ in $\P^N_k$.  Thus $\XX_0$ is reduced and
irreducible, but it may be quite singular.

Lipman's general result on desingularizing 2-dimensional schemes (see
\cite{Lipman78} or \cite{Artin86b}) can be used to find a non-singular
model of $\XX_0$.  More precisely, normalizing $\XX_0$ results in a
scheme with isolated singularities.  Let $\XX_1$ be the result of
blowing up the normalization of $\XX_0$ once at each (reduced) closed
singular point.  Now inductively let $\XX_n$ ($n\ge2$) be the result
of normalizing $\XX_{n-1}$ and blowing up each singular point.
Lipman's theorem is that the sequence $\XX_n$ yields a regular scheme
after finitely many steps.  The resulting regular projective scheme
$\XX_n$ is equipped with a rational map to $\CC$.

After further blowing up and/or blowing down, we arrive at an
irreducible, regular, 2-dimensional scheme projective over $k$ with a
projective, relatively minimal morphism $\XX\to\CC$ whose generic
fiber is $X\to\spec K$.  Here relatively minimal means that if $\XX'$
is regular with a proper morphism $\XX'\to\CC$, and if there is a
factorization $\XX\labeledto{f}\XX'\to\CC$ with $f$ a birational
proper morphism, then $f$ is an isomorphism.  This is equivalent to
there being no $(-1)$-curves in the fibers of $\pi$.

The uniqueness of $\XX$ (subject to the properties) follows from
\cite[Thm.~4.4]{Lichtenbaum68}.

Since $\CC$ is a smooth curve, $\XX$ is irreducible, and $\pi$ is
dominant, $\pi$ is flat.
 
If $k'$ is an extension of $k$ over which $\XX$ becomes reducible,
then every component of $\XX\times_kk'$ dominates $\CC\times_kk'$.
(This is because $\pi$ is flat, so $\XX\times_kk'\to\CC\times_kk'$ is
flat.)  In this case, $X$ would be reducible over $k'K$.  But we assumed
that $X$ is absolutely irreducible, so $\XX$ must also be absolutely
irreducible.

By construction $\XX$ is projective over $k$.

Let $\CC'$ be 
${\bf\spec}_{\OO_\CC}\pi_*\OO_{\XX}$
(global spec) so
that the Stein factorization of $\pi$ is $\XX\to\CC'\to\CC$.  Then
$\CC'\to\CC$ is finite and flat.  Let $\CC'_\eta=\spec L\to\eta=\spec
K$ be the generic fiber of $\CC'\to\CC$.  Then the algebra $L$ is
finite and thus algebraic over $K$.  On the other hand, $L$ is
contained in $k(\XX)$ and $K$ is algebraically closed in $k(\XX)$, so
we have $L=K$.  Thus $\CC'\to\CC$ is finite flat of degree 1, and
since $\CC$ is smooth, it must be an isomorphism.  This proves that
$\pi_*\OO_\XX=\OO_\CC$.  It follows (e.g., by
\cite[III.11.3]{HartshorneAG}) that the fibers of $\pi$ are all
connected.

This completes the proof of the proposition.
\end{proof}

For the rest of the paper, $\pi:\XX\to\CC$ will be the fibration
constructed in the proposition.  We will typically extend $k$ if
necessary so that $\XX$ has a $k$-rational point.

\section{Smoothness and regularity}\label{s:smoothness-regularity}

If $k$ is perfect, then because $\XX$ is regular, $\XX\to \spec k $ is
automatically smooth.  However, it need not be smooth if $k$ is not
perfect.

Let us consider an example.  Since smoothness and regularity are local
properties, we give an affine example and leave to the reader the easy
task of making it projective.  Let $\F$ be a field of characteristic
$p>0$ and let $\tilde\XX$ be the closed subset of $\A^4_\F$ defined by
$y^2+xy-x^3-(t^p+u)$.  The projection $(x,y,t,u)\mapsto(x,y,t)$
induces an isomorphism $\tilde\XX\to\A^3$ and so $\tilde\XX$ is a
regular scheme.  Let $k=\F(u)$ and let $\XX\to\spec k$ be the generic
fiber of the projection $\tilde\XX\to\A^1$, $(x,y,t,u)\mapsto u$.
Since the local rings of $\XX$ are also local rings of $\tilde\XX$,
$\XX$ is a regular scheme.  On the other hand, $\XX\to\spec k$ is not
smooth at the point $x=y=t=0$.  Now let $X$ be the generic fiber of
the projection $\XX\to\A^1$, $(x,y,t)\mapsto t$.  Then $X$ is the
affine scheme $\spec K[x,y]/(y^2+xy-x^3-(t^p-u))$ and this is easily
seen to be smooth over $K$.

This shows that the minimal regular model $\XX\to\CC$ of $X\to\spec K$
need not be smooth over $k$.

If we are given $X\to\spec K$ and the regular model $\XX$ is not
smooth over $k$, we can remedy the situation by extending $k$.
Indeed, let $\kbar$ denote the algebraic closure of $k$, and let 
$\overline\XX\to\overline\CC$ be the regular model of
$X\times_k\kbar\to K\kbar$.  Then since 
$\kbar$ is perfect, $\overline\XX$ is smooth over $\kbar$.  
It is clear that there is a
finite extension $k'$ of $k$ such that $\XXbar$ is defined over $k'$
and birational over $k'$ to $\XX$.  So replacing $k$ by $k'$ we that
the regular model $\XX\to\CC$ of $X\to\spec K$ is smooth over $k$.

We will generally assume below that $\XX$ is smooth over $k$.

\subsection{Correction to \cite{Ulmer11}}
In \cite{Ulmer11}, Lecture~2, \S2, p.~237, speaking of a
2-dimensional, separated, reduced scheme of finite type over a field
$k$, we say ``Such a scheme is automatically quasi-projective and is
projective if and only if it is complete.''  This is not correct in
general---we should also assume that $\XX$ is non-singular.  In fact,
when the ground field is finite, it suffices to assume that $\XX$ is
normal.  (See Fontana, \emph{Compositio Mathematica\/}, {\bf 31},
1975.)

\section{Structure of fibers}\label{s:fibers}
We write $\XX_v$ for the fiber of $\pi$ over the closed point $v$ of
$\CC$.  We already noted that the fibers $\XX_v$ are connected.

We next define certain multiplicities of components, following
\cite[Ch.~9]{BoschLutkebohmertRaynaudNM}.  Let $\XX_{v,i}$,
$i=1,\dots,r$ be the reduced irreducible components of $\XX_v$ and let
$\eta_{v,i}$ be the corresponding generic points of $\XX_v$.  Let
$\kbar_v$ be an algebraic closure of the residue field at $v$ and
write $\overline \XX_v$ for $\XX_v\times_{k_v}\kbar_v$ and
$\overline\eta_{v,i}$ for a point of $\overline\XX_v$ over
$\eta_{v,i}$.  We define the {\it multiplicity\/} of $X_{v,i}$ in
$X_v$ to be the length of the Artin local ring $\OO_{X_v,\eta_{v,i}}$,
and the {\it geometric multiplicity\/} of $X_{v,i}$ in $X_v$ to be the
length of $\OO_{\overline\XX_v,\overline\eta_{v,i}}$.  The geometric
multiplicity is equal to the multiplicity when the characteristic of
$k$ is zero and it is a power of $p$ times the multiplicity when the
characteristic of $k$ is $p>0$.

We write $\XX_v=\sum_{i}m_{v,i}\XX_{v,i}$ where $m_{v,i}$ is the
multiplicity of $\XX_{v,i}$ in $\XX_v$.  This is an equality of
Cartier divisors on $\XX$.  We define the {\it multiplicity\/} $m_v$
of the fiber $\XX_v$ to be the gcd of the multiplicities $m_{v,i}$.

The multiplicity $m_v$ divides the gcd of the geometric multiplicities
of the components of $\XX_v$ which in turn divides the index
$\delta_v$ of $\XX_v$.  In particular, if $X$ has a $K$-rational point
(so that $\XX\to\CC$ has a section) then for every $v$ we have
$m_v=1$.

We now turn to the combinatorial structure of the fiber $\XX_v$.  A
convenient reference for what follows is \cite[Ch.~9]{LiuAGAC}.

We write $D.D'$ for the intersection multiplicity of two divisors on
$\XX$.  It is known \cite[9.1.23]{LiuAGAC} that the intersection form
restricted to the divisors supported in a single fiber $\XX_v$ is
negative semi-definite, and that its kernel consists exactly of the
divisors which are rational multiples of the entire fiber.  (Thus if
the multiplicity of the fiber $\XX_v$ is $m_v$, then
$(1/m_v)\XX_v:=\sum_i(m_{v,i}/m_v)\XX_{v,i}$ generates the kernel of
the pairing.)

It is in principle possible to use this result to give a
classification of the possible combinatorial types of fibers (genera
and multiplicities of components, intersection numbers) for a fixed
value of $g_X$.  Up to a suitable equivalence relation, the set of
possibilities for a given value of $g_X$ is finite
\cite{ArtinWinters71}.  When $X$ is an elliptic curve and the residue
field is assumed perfect, this is the well-known Kodaira-N\'eron
classification.  For higher genus, the situation rapidly becomes
intractable.  We note that the list of possibilities can be strictly
longer when one does not assume that the residue field $k_v$ is
perfect.  See \cite{Szydlo04} for a complete analysis of the case
where $X$ is an elliptic curve.

\section{Leray spectral sequence}\label{s:Leray}
Fix a prime $\ell\neq\ch(k)$.  We consider the Leray spectral
sequence for $\pi:\XXbar\to\CCbar$ in $\ell$-adic cohomology.  The
$E_2$ term
$$E_2^{pq}=H^p(\CCbar,R^q\pi_*\Ql)$$
vanishes outside the range $0\le p,q\le2$ and so the only possibly
non-zero differentials are $d_2^{01}$ and $d_2^{02}$.  We will show
that these both vanish and so the sequence degenerates at $E_2$.

The differential $d_2^{01}$ sits in an exact sequence of low degree
terms that includes
$$H^0(\CCbar,R^1\pi_*\Ql)\overset{d_2^{01}}{\longrightarrow}
H^2(\CCbar,\pi_*\Ql)\to H^2(\XXbar,\Ql).$$ 
Now $\pi_*\Ql=\Ql$, and the edge morphism $H^2(\CCbar,\Ql)\to
H^2(\XXbar,\Ql)$ is simply the pull-back $\pi^*$.  If $i:\DD\into\XX$
is a multisection of degree $n$ and $j=\pi i$, then we have a
factorization of $j^*$:
$$H^2(\CCbar,\Ql)\to H^2(\XXbar,\Ql)\to H^2(\DDbar,\Ql)$$
and we have a trace map
$$H^2(\DDbar,\Ql)\to H^2(\CCbar,\Ql)$$
for the finite morphism $j:\DD\to\CC$.  The composition $j^*$ followed
by the trace map is just multiplication by $n$, which is injective,
and therefore $H^2(\CCbar,\Ql)\to H^2(\XXbar,\Ql)$ is also injective,
which implies that $d_2^{01}=0$.

Now consider $d_2^{02}$, which sits in an exact sequence
$$H^2(\XXbar,\Ql)\to
H^0(\CCbar,R^2\pi_*\Ql)\overset{d_2^{02}}{\longrightarrow}
H^2(\CCbar,R^1\pi_*\Ql).$$ 
We can deduce that it too is zero by using duality, or by the
following argument, which unfortunately uses terminology not
introduced until Chapter~\ref{ch:cohom}.  The careful reader will have
no trouble checking that there is no circularity.

Here is the argument: Away from the reducible fibers of $\pi$, the
sheaf $R^2\pi_*\Ql$ is locally constant of rank 1.  At a closed point
$\vbar$ of $\CCbar$ where the fiber of $\pi$ is reducible, the stalk
of $R^2\pi_*\Ql$ has rank $f_\vbar$, the number of components in the
fiber.  The cycle class of a component of a fiber in $H^2(\XXbar,\Ql)$
maps onto the corresponding section in $H^0(\CCbar,R^2\pi_*\Ql)$ and
so $H^2(\XXbar,\Ql)\to H^0(\CCbar,R^2\pi_*\Ql)$ is surjective.  This
implies that $d_2^{02}=0$ as desired.

For later use, we record the exact sequence of low degree terms
(where the zero on the right is because $d_2^{01}=0$):
\begin{equation}\label{eq:H1s}
0\to H^1(\CCbar,\Ql)\to H^1(\XXbar,\Ql)\to H^0(\CCbar,R^1\pi_*\Ql)\to0.
\end{equation}

We end this section by noting a useful property of $R^1\pi_*\Ql$
when $k$ is finite.

\begin{prop}\label{prop:middle-extension}
  Suppose that $k$ is finite, $\XX$ is a smooth, proper surface over
  $k$ equipped with a flat, generically smooth, proper morphism
  $\pi:\XX\to\CC$.  Let $\ell$ be a prime not equal to the
  characteristic of $k$ and let $\FF=R^1\pi_*\Ql$, a constructible
  $\ell$=adic sheaf on $\CC$.  If $j:\eta\into\CC$ is the inclusion of
  the generic point, then the canonical morphism $\FF\to j_*j^*\FF$ is
  an isomorphism.
\end{prop}

\begin{rem}
The same proposition holds if we let $j:U\into\CC$ be the inclusion of a
non-empty open subset over which $\pi$ is smooth.  Sheaves with this
property are sometimes called ``middle extension'' sheaves.  It is
also useful to note that for $\FF$ and $j$ as above, we have
an isomorphism 
$$H^1(\CCbar,\FF)\cong\im\left(H^1_c(\overline{U},j^*\FF)\to
  H^1(\overline{U},j^*\FF)\right)$$
(image of compactly supported cohomology in usual cohomology).  This
is ``well known'' but the only reference I know is \cite[7.1.6]{Ulmer05}.
\end{rem}

\begin{proof}[Proof of Proposition~\ref{prop:middle-extension}]
  We will show that for every geometric point $\xbar$ over a closed
  point $x$ of $\CC$, the stalks of $\FF$ and $j_*j^*\FF$ are
  isomorphic.  (Note that the latter is the group of invariants of
  inertia at $x$ in the Galois module $\FF_{\etabar}$.)

  The local invariant cycle theorem (Theorem 3.6.1 of
  \cite{Deligne80}) says that the map of stalks
  $\FF_\xbar\to(j_*j^*\FF)_\xbar$ is surjective.  A closer examination
  of the proof shows that it is also injective in our situation.
  Indeed, in the diagram (8) on p.~214 of \cite{Deligne80}, the group
  on the left $H^0(X_{\etabar})_I(-1)$ is pure of weight 2, whereas
  $H^1(X_s)$ is mixed of weight $\le1$; also the preceding term in the
  vertical sequence is (with $\Ql$ coefficients) dual to $H^3(X_s)(2)$
  and this vanishes for dimension reasons.  Thus the diagonal map
  $sp^*$ is also injective.

  (Note the paragraph after the display (8): this argument only works
  with $\Ql$ coefficients, not necessarily with $\Zl$ coefficients.)
\end{proof}

See \cite[Prop.~7.5.2]{KatzMMP} for a more general result with a
similar proof.

We have a slightly weaker result in integral cohomology.

\begin{cor}\label{cor:R1H1-cohom}
Under the hypotheses of the proposition, the natural map
$$R^1\pi_*\Zl\to j_*j^*R^1\pi_*\Zl$$
has kernel and cokernel supported at finitely many closed points of
$\CC$ and with finite stalks.  Thus the induced map
$$H^i(\CCbar,R^1\pi_*\Zl)\to H^i(\CCbar,j_*j^*R^1\pi_*\Zl)$$
is an isomorphism for $i>1$ and surjective with finite kernel for
$i=1$.
\end{cor}

\begin{proof}
  The proof of the Proposition shows that the kernel and cokernel are
  supported at points where $\pi$ is not smooth, a finite set of
  closed points. Also, the Proposition shows that the stalks are
  torsion.  Since they are also finitely generated, they must be
  finite.
\end{proof}

\section{Cohomological flatness}

Since $\pi:\XX\to\CC$ is flat, general results on cohomology and base
change (e.g., \cite[III.12]{HartshorneAG}) imply that the coherent
Euler characteristic of the fibers of $\pi$ is constant, i.e., the
function $v\mapsto\chi(\XX_v,\OO_{\XX_v})$ is constant on $\CC$.
Moreover, the dimensions of the individual cohomology groups
$h^i(\XX_v,\OO_{\XX_v})$ are upper semi-continuous.  They are not in
general locally constant.

To make this more precise, we recall a standard exact sequence from
the theory of cohomology and base change:
$$0\to(R^i\pi_*\OO_\XX)\tensor_{\OO_{\CC}}\kappa(v)\to H^i(\XX_v,\OO_{\XX_v})
\to 
(R^{i+1}\pi_*\OO_\XX)[\varpi_v]
\to0.$$ Here $\XX_v$ is the fiber of
$\pi$ at $v$, the left-hand group is the fiber of the coherent sheaf
$R^i\pi_*\OO_\XX$ at $v$, and the right-hand group is the
$\varpi_v$-torsion in $R^{i+1}\pi_*\OO_\XX$ where $\varpi_v$ is a
generator of $\m_v$.  Since $R^i\pi_*\OO_\XX$ is coherent, the
function
$$v\mapsto\dim_{\kappa(v)}(R^i\pi_*\OO_\XX)\tensor_{\OO_{\CC}}\kappa(v)$$
is upper semi-continuous, and it is locally constant if and only if
$R^i\pi_*\OO_\XX$ is locally free.  Thus the obstruction to $v\mapsto
h^i(\XX_v,\OO_{\XX_v})$ being locally constant is controlled by
torsion in $R^i\pi_*\OO_\XX$ and $R^{i+1}\pi_*\OO_\XX$.

We say that $\pi$ is {\it cohomologically flat in dimension $i$\/} if
formation of $R^i\pi_*\OO_\XX$ commutes with arbitrary change of base,
i.e., if for all $\phi:T\to\CC$, the base change morphism
$$\phi^*R^i\pi_*\OO_\XX\to R^i\pi_T\OO_{\T\times_\CC\XX}$$
is an isomorphism.  Because the base $\CC$ is a smooth curve, this is
equivalent to the same condition where $T\to\CC$ runs through
inclusions of closed point, i.e., to the condition that
$$(R^i\pi_*\OO_\XX)\tensor_{\OO_{\CC}}\kappa(v)\to
H^i(\XX_v,\OO_{\XX_v})$$
be an isomorphism for all closed points $v\in\CC$.  By the exact
sequence above, this is equivalent to $R^{i+1}\pi_*\OO_\XX$ being
torsion free, and thus locally free.

Since $\pi:\XX\to\CC$ has relative dimension 1,
$R^{i+1}\pi_*\OO_\XX=0$ for $i\ge1$ and $\pi$ is automatically
cohomologically flat in dimension $i\ge1$.  It is cohomologically flat
in dimension 0 if and only if $R^1\pi_*\OO_\XX$ is locally free.
To lighten terminology, in this case we say simply that $\pi$ is
cohomologically flat.

Since $\pi_*\OO_\XX=\OO_\CC$ is free of rank 1, we have that $\pi$ is
cohomologically flat if and only if $v\mapsto h^i(\XX_v,\OO_{\XX_v})$
is locally constant (and thus constant since $\CC$ is connected) for
$i=0,1$.  Obviously the common value is 1 for $i=0$ and $g_X$ for
$i=1$.

Raynaud gave a criterion for cohomological flatness in
\cite[7.2.1]{Raynaud70}.  Under our hypotheses ($\XX$ regular, $\CC$ a
smooth curve over $k$, and $\pi:\XX\to\CC$ proper and flat with
$\pi_*\OO_\XX=\OO_\CC$), $\pi$ is cohomologically flat if $\ch(k)=0$
or if $k$ is perfect and the following condition holds: For each
closed point $v$ of $\CC$, let $d_v$ be the $\gcd$ of the geometric
multiplicities of the components of $\XX_v$.  The condition is that
$d_v$ is prime to $p=\ch(k)$ for all $v$.

Raynaud also gave an example of non-cohomological flatness which we
will make more explicit below.  Namely, let $S$ be a complete DVR with
fraction field $F$ and algebraically closed residue field of
characteristic $p>0$, and let $\EE$ be an elliptic curve over $S$ with
either multiplicative or good supersingular reduction.  Let $Y$ be a
principal homogeneous space for $E=\EE\times \spec F$ of order $p^e$
($e>0$) and let $\YY$ be a minimal regular model for $Y$ over $S$.
Then $\YY\to\spec S$ is not cohomologically flat.  Moreover, the
invariant $\delta$ defined as above is $p^e$.  It follows from later
work \cite[Thm.~6.6]{LiuLorenziniRaynaud04} that the special fiber of
$\YY$ is like that of $\EE$, but with multiplicity $p^e$.  The
explicit example below should serve to make the meaning of this clear.

\section{Example}\label{s:CohomFlatnessExample}
Let $k=\F_2$ and $\CC=\P^1$ so that $K=\F_2(t)$.  We give an example
of a curve of genus 1 over $\A^1_K$ which is not cohomologically flat
at $t=0$.

Consider the elliptic curve $E$ over $K$ given by
$$y^2+xy=x^3+tx.$$
The point $P=(0,0)$ is 2-torsion.  The discriminant of this model is
$t^2$ so $E$ has good reduction away from the place $t=0$ of $K$.  At
$t=0$, $E$ has split multiplicative reduction with minimal regular
model of type $I_2$.

The quotient of $E$ by the subgroup generated by $P$ is $\phi:E\to
E'$, where $E'$ is given by
$$s^2+rs=r^3+t.$$
One checks that the degree of the conductor of $E'$ is 4 and so 
(by \cite{Ulmer11}, Lecture~1, Theorem~9.3 and Theorem~12.1(1))
the rank of $E'(K)$ is zero.  Also, $E'(K)$ has no 2-torsion.  

Therefore, taking cohomology of the sequence 
$$0\to E[\phi]\to E\to E'\to 0,$$
we find that the map
$$K/\wp(K)\cong H^1(K,\Z/2\Z)\cong H^1(E,E[\phi])\to H^1(K,E)$$
is injective.  For $f\in K$, we write $X_f$ for the torsor for $E$
obtained from the class of $f$ in $K/\wp(K)$ via this map.

Let $L$ be the quadratic extension of $K$ determined by $f$, i.e., 
$$L=K[u]/(\wp(u)-f).$$  
The action of $G=\gal(L/K)=\Z/2\Z$ on $L$ is $u\mapsto u+1$.  Let $G$
act on $K(E)$ by translation by $P$.  Explicitly, one finds
$$(x,y)\mapsto (t/x,t(x+y)/x^2)$$
and so $y/x\mapsto y/x+1$.

The function field of $X_f$ is the field of $G$-invariants in $L(E)$
where $G$ acts as above.  One finds that 
$$K(X_f)=\frac{K(r)[s,z]}{\left(s^2+rs+r^3+t,z^2+z+r+f\right)}$$
which presents $X_f$ as a double cover of $E'$.

Let us now specialize to $f=t^{-1}$ and find a minimal regular model of
$X_f$ over $R=\F_2[t]$.  Let
$$\UU=\spec \frac{R[r,s,w]}{\left(s^2+rs+r^3+t,w^2+tw+t^2r+t\right)}.$$ 
Then $\UU$ is regular away from $t=r=s=w=0$ and its generic fiber is
isomorphic (via $w=tz$) to an open subset of $X_f$. Let
$$\VV=\spec
\frac{R[r',s',w']}{\left(s'+r's'+r^{\prime3}+ts^{\prime3},w^{\prime2}+tw'+t^3r's'+t\right)}$$ 
which is a regular scheme whose generic fiber is another open subset
of $X_f$.  Let $\YY_f$ be the result of glueing $\UU$ and $\VV$ via 
$(r,s,w)=(r'/s',1/s',w'+tr^{\prime2}/s')$.  The generic fiber of
$\YY_f$ is $X_f$ and $\YY_f$ is regular away from $r=s=w=t=0$.  Note
that the special fiber of $\YY_f$ is isomorphic to the product of the
doubled point $\spec\F_2[w]/(w^2)$ with the projective nodal plane cubic
$$\proj\frac{\F_2[r,s,v]}{(s^2v+rsv+r^3)}.$$
In particular, the special fiber over $t=0$, call it
$\YY_{f,0}$, satisfies
$H^0(\YY_{f,0},\OO_{\YY_{f,0}})=\F_2[w]/(w^2)$.  
This
shows that
$\YY_f$ is not cohomologically flat over $\CC$ at $t=0$.

To finish the example, we should blow up $\YY_f$ at its unique
non-regular point.  The resulting scheme $\XX_f$ is regular and flat
over $R$, but it is not cohomologically flat at $t=0$.  The fiber over
$t=0$ is the product of a doubled point and a N\'eron configuration of
type $I_2$, and its global regular functions are $\F_2[w]/(w^2)$.
This is in agreement with \cite[6.6]{LiuLorenziniRaynaud04}.

We will re-use parts of this example below.

\begin{exer}
  By the earlier discussion, $R^1\pi_*\OO_{\XX_f}$ has torsion
  at $t=0$.  Make this explicit.
\end{exer}

\chapter{Properties of $J_X$}\label{ch:J}

\section{Review of Picard functors}\label{s:pic}
We quickly review basic material on the relative Picard functor.  The
original sources \cite{GrothendieckFGA} and \cite{Raynaud70} are still
very much worth reading.  Two excellent modern references with more
details and historical comments are \cite{Kleiman05} and
\cite{BoschLutkebohmertRaynaudNM}.

\subsection{The relative Picard functor}
For any scheme $\YY$, we write $\Pic(\YY)$ for the {\it Picard
  group\/} of $\YY$, i.e., for the group of isomorphism classes of
invertible sheaves on $\YY$.  This group can be calculated
cohomologically: $\Pic(\YY)\cong H^1(\YY,\OO_\YY^\times)$ (cohomology
computed in the Zariski, \'etale, or finer topologies).

Now fix a morphism of schemes $f:\YY\to\SS$ 
(separated and of finite type, as always).
If $\TT\to\SS$ is a morphism of schemes, we write
$\YY_\TT$ for $\YY\times_\SS\TT$ and $f_\TT$ for the projection
$\YY_\TT\to\TT$.  Define a functor $P_{\YY/\SS}$ from schemes over
$\SS$ to abelian groups by the rule
$$\TT\mapsto P_{\YY/\SS}(\TT)=\frac{\Pic(\YY_\TT)}{f_\TT^*\Pic(\TT)}.$$

We define the {\it relative Picard functor\/} $\Picf_{\YY/\SS}$ to be
the fppf sheaf associated to $P_{\YY/\SS}$.  (Here ``fppf'' means
``faithfully flat and finitely presented''.  See
\cite[8.1]{BoschLutkebohmertRaynaudNM} for details on the process of
sheafification.)  Explicitly, if $\TT$ is affine, an element of
$\Picf_{\YY/\SS}(\TT)$ is represented by a line bundle
$\xi'\in\Pic(\YY\times_\SS\TT')$ where $\TT'\to\TT$ is fppf, subject
to the condition that there should exist an fppf morphism
$\tilde\TT\to\TT'\times_\TT\TT'$ such that the pull backs of $\xi'$
via the two projections
$$\tilde\TT\to\TT'\times_\TT\TT'\rightrightarrows\TT'$$
are isomorphic.  Two such elements $\xi_i\in\Pic(\YY\times_\SS\TT'_i)$
represent the same element of $\Picf_{\YY/\SS}(\TT)$ if and only if
there is an fppf morphism 
$\tilde\TT\to\TT'_1\times_\TT\TT'_2$
such that the pull-backs of the $\xi_i$ to $\tilde\TT$ via the two
projections are isomorphic.  Fortunately, under mild hypotheses, this
can be simplified quite a bit!

Assume that $f_*\OO_\YY=\OO_\SS$ universally.  This means that for
all $\TT\to\SS$, we have $f_{\TT*}\OO_{\YY_\TT}=\OO_\TT$.
Equivalently, $f_*\OO_\YY=\OO_\SS$ and $f$ is cohomologically flat in
dimension 0.  In this case, for all $\TT\to\SS$, we have an exact
sequence
\begin{equation}\label{eq:low-degree}
0\to\Pic(\TT)\to\Pic(\YY_\TT)\to\Picf_{\YY/\SS}(\TT)
\to\Br(\TT)\to\Br(\YY_\TT).
\end{equation}
(This is the exact sequence of low-degree terms in the Leray spectral
sequence for $f:\YY_\TT\to\TT$, computed with respect to the fppf
topology.)  Here $\Br(\TT)=H^2(\TT,\OO_\TT^\times)$ and
$\Br(\YY_\TT)=H^2(\YY_\TT,\OO_{\YY_\TT}^\times)$ are the cohomological
Brauer groups of $\TT$ and $\YY_\TT$, again computed with the fppf
topology.  (It is known that the \'etale topology gives the same
groups.)  See \cite[8.1]{BoschLutkebohmertRaynaudNM} for the
assertions in this paragraph and the next.

In case $f$ has a section, we get a short exact sequence
$$0\to\Pic(\TT)\to\Pic(\YY_\TT)\to\Picf_{\YY/\SS}(\TT)\to0$$
and so in this case 
$$\Picf_{\YY/\SS}(T)=P_{\YY/\SS}(T)=\frac{\Pic(\YY_\TT)}{f_\TT^*\Pic(\TT)}.$$

\subsection{Representability and $\Pic^0$ over a field}\label{ss:Pic/field}
The simplest representability results will be sufficient for many of
our purposes.  

To say that $\Picf_{\YY/\SS}$ is represented by a scheme
$\Pic_{\YY/\SS}$ means that for all $S$-schemes $\TT$,
$$\Picf_{\YY/\SS}(\TT)=\Pic_{\YY/\SS}(\TT)=\mor_\SS(\TT,\Pic_{\YY/\SS}).$$

Suppose that $\SS$ is the spectrum of a field and that $\YY\to\SS$ is
proper.  Then $\Picf_{\YY/\SS}$ is represented by a scheme
$\Pic_{\YY/\SS}$ which is locally of finite type over $\SS$ \cite[8.2,
Thm.~3]{BoschLutkebohmertRaynaudNM}.  The connected component of this
group scheme will be denoted $\Pic^0_{\YY/\SS}$.  If $\YY\to\SS$ is
smooth and geometrically irreducible, then $\Pic^0_{\YY/\SS}$ is
proper \cite[8.4, Thm.~3]{BoschLutkebohmertRaynaudNM}.

The results of the previous paragraph apply in particular to $X/K$ and
$\XX/k$.  Moreover, since $X/K$ is a curve, $H^2(X,\OO_X)=0$ and so
$\Pic^0_{X/K}$ is smooth and hence an abelian variety \cite[8.4,
Prop.~2]{BoschLutkebohmertRaynaudNM} (plus our assumption that
morphisms are of finite type to convert formal smoothness into
smoothness) or \cite[9.5.19]{Kleiman05}.

In general $\Pic_{\XX/k}$ need not be reduced and so need not be
smooth over $k$.  If $k$ has characteristic zero or
$H^2(\XX,\OO_\XX)=0$, then $\Pic^0_{\XX/k}$ is again an abelian
variety.  We define
$$\Picvar_{\XX/k}=\left(\Pic^0_{\XX/k}\right)_{red}$$
the {\it Picard variety\/} of $\XX$, which is an abelian variety.
See \cite{Serre58} and \cite{MumfordLCAS} for an analysis of
non-reduced Picard schemes and \cite{Liedtke09} for more in the case
of surfaces.

Since we have assumed that $k$ is large enough that $\XX$ has a
rational point, for all $k$-schemes $\TT$, we have
$$\Pic^0_{\XX/k}(\TT)=\frac{\Pic^0(\XX_\TT)}{\pi_T^*\Pic^0(\TT)}.$$

\subsection{More general bases}
For any $\YY\to\SS$ which is proper, define $\Picf^0_{\YY/\SS}$ to be
the subfunctor of $\Picf_{\YY/\SS}$ consisting of elements whose
restrictions to fibers $\YY_s$, $s\in\SS$, lie in
$\Pic^0_{\YY_s/\kappa(s)}$.  

We need a deeper result to handle $\XX\to\CC$.  Namely, assume that
$\YY$ is regular, $\SS$ is 1-dimensional and regular, $f:\YY\to\SS$ is
flat, and projective of relative dimension 1 such that
$f_*\OO_\SS=\OO_\CC$.  Over the open subset of $\SS$ where $f$ is
smooth, $\Picf_{\YY/\SS}^0$ is represented by an abelian scheme
\cite[9.4 Prop.~4]{BoschLutkebohmertRaynaudNM}.  

Over all of $\SS$, we cannot hope for reasonable representability
results unless we make further hypotheses on $f$.  If we assume that
each fiber of $f$ has the property that the gcd of the geometric
multiplicities of its irreducible components is 1, then
$\Picf_{\YY/\SS}^0$ is represented by a separated $\SS$-scheme
\cite[9.4 Thm.~2]{BoschLutkebohmertRaynaudNM}.

\subsection{Example}
Here is an explicit example where an element of $\Pic_{X/K}(K)$ is not
represented by an element of $\Pic(X)$.

Let $E/K$ be the elliptic curve of
Section~\ref{s:CohomFlatnessExample} and let $X_f$ be the homogeneous
space for $E$ as above, where $f\in K$ will be selected later.  Then
$E$ is the Jacobian of 
$X$, and so by the results quoted in Subsection~\ref{ss:Pic/field},
$E$ represents the functor
$\Picf_{X/K}^0$.  Let us consider the point $P=(0,0)$ above in
$E(K)=\Pic^0_{X/K}(K)$ and show that for many choices of $f$, this
class is not represented by an invertible sheaf on $X=X_f$.
Equivalently we want to show that $P$ does not go to 0 in $\Br(K)$ (in
the sequence~\eqref{eq:low-degree} above).

The image of $P$ in $\Br(K)$ is given by a pairing studied by
Lichtenbaum \cite{Lichtenbaum69}.  More precisely, consider the image
of $P$ under the coboundary
$$E(K)\to H^1(K,E'[\phi^\vee])\cong H^1(K,\mu_2)\cong
K^\times/K^{\times2}.$$ 
By Kramer's results on 2-descent in characteristic 2
\cite[1.1b]{Kramer77}, the image of $P$ is the class of $t$.  On the
other hand, suppose that $X=X_f$ is the torsor for $E$ corresponding
to $f\in K/\wp(K)$.  Then the local invariant at $v$ of the image of
$P$ in $\Br(K)$ is given by
$$\frac12\res_v\left(f\frac{dt}t\right)\in\frac12\Z/\Z\subset\Br(K_v).$$
Thus, for example, if we take $f=1/(t-1)$, then we get an element of
the Brauer group ramified at 1 and $\infty$ and so the class of $P$
does not come from $\Pic(X_f)$.  


\section{The Jacobian}
By definition, the Jacobian $J_X$ of $X$ is the group scheme
$\Pic^0_{X/K}$.  Because $X$ is a smooth, projective curve, $J_X$ is
an abelian variety of dimension $g_X$ where $g_X$ is the genus of $X$,
and it is equipped with a canonical principal polarization given by
the theta divisor.  We refer to \cite{Milne86jv} for more on the basic
properties of $J_X$, including the Albanese property and autoduality.

\section{The $K/k$ trace of $J_X$}
Recall that $(B,\tau)$ denotes the $K/k$ trace of $J_X$, which by
definition is a final object in the category of pairs $(A,\sigma)$
where $A$ is a $k$-abelian variety and $\sigma:A\times_kK\to J_X$ is a
$K$-morphism of abelian varieties.

We refer to \cite{Conrad06} for a discussion in modern language of the
existence and basic properties of $(B,\tau)$.  In particular, Conrad
proves that $(B,\tau)$ exists, it has good base change properties, and
(when $K/k$ is regular, as it is in our case) $\tau$ has finite
connected kernel with connected Cartier dual.  In particular, $\tau$
is purely inseparable.

\section{The Lang-N\'eron theorem}\label{s:LangNeron}
Define the Mordell-Weil group $\MW(J_X)$ as
$$\MW(J_X)=\frac{J_X(K)}{\tau B(k)}$$
where as usual $(B,\tau)$ is the $K/k$ trace of $J_X$.  (In view of
the theorem below, perhaps this would be better called the
Lang-N\'eron group.)

Generalizing the classical Mordell-Weil theorem, we have the following
finiteness result, proven independently by Lang and N\'eron.

\begin{thm}
$\MW(J_X)$ is a finitely generated abelian group.
\end{thm}

Note that when $k$ is finitely generated, $B(k)$ is finitely generated
as well, and so $J_X(K)$ is itself a finitely generated abelian
group. For large fields $k$, $B(k)$ may not be finitely generated, and
so it really is necessary to quotient by $\tau B(k)$ in order to get a
finitely generated abelian group.

We refer to \cite{Conrad06} for a proof of the theorem in modern
language.  Roughly speaking, the proof there follows the general lines
of the usual proof of the Mordell-Weil theorem for an abelian variety
over a global field:  One shows that $\MW(J_X)/n$ is finite by embedding
it in a suitable cohomology group, and then uses a theory of heights to
deduce finite generation of $\MW(J_X)$.  

Another proof proceeds by relating $\MW(J_X)$ to the N\'eron-Severi
group $\NS(\XX)$ (this is the Shioda-Tate isomorphism discussed in the
next chapter) and then proving that $\NS(\XX)$ is finitely generated.
The latter was proven by Kleiman in \cite[XIII]{SGA6}.

\chapter{Shioda-Tate and heights}\label{ch:ST}

\section{Points and curves}
We write $\divr(\XX)$ for the group of (Cartier or Weil) divisors on
$\XX$ and similarly with $\divr(X)$.  A prime divisor on $\XX$ is
\emph{horizontal} if it is flat over $\CC$ and \emph{vertical} if it
is contained in a fiber of $\pi$.  The group $\divr(\XX)$ is the
direct sum of its subgroups $\divr^{hor}(\XX)$ and $\divr^{vert}(\XX)$
generated respectively by horizontal and vertical prime divisors.

Restriction of divisors to the generic fiber of $\pi$ induces a
homomorphism
$$\divr(\XX)\to\divr(X)$$
whose kernel is $\divr^{vert}(\XX)$ and which induces an isomorphism
$\divr^{hor}(\XX)\cong\divr(X)$.  The inverse of this isomorphism
sends a closed point of $X$ to its scheme-theoretic closure in $\XX$.

We define a filtration of $\divr(\XX)$ by declaring that
$L^1\divr(\XX)$ be the subgroup of $\divr(\XX)$ consisting of divisors
whose restriction to $X$ has degree 0, and by declaring that
$L^2\divr(\XX)=\divr^{vert}(\XX)$.

We define $L^i\Pic(\XX)$ to be the image of $L^i\divr(\XX)$ in
$\Pic(\XX)$.  Also, recall that $\Pic^0(\XX)=\Pic^0_{\XX/k}(k)$ is the
group of invertible sheaves which are algebraically equivalent to zero
(equivalence over $\kbar$ as usual).

Recall that the N\'eron-Severi group $\NS(\XXbar)$ is
$\Pic(\XXbar)/\Pic^0(\XXbar)$ and $\NS(\XX)$ is by definition the
image of $\Pic(\XX)$ in $\NS(\XXbar)$.  We define $L^i\NS(\XX)$ as the
image of $L^i\Pic(\XX)$ in $\NS(\XX)$.

It is obvious that $\NS(\XX)/L^1\NS(\XX)$ is an infinite cyclic group;
as generator we may take the class of a horizontal divisor of total
degree $\delta$ over $\CC$, where $\delta$ is the index of $X$.

Recall that the intersection pairing restricted to the group generated
by the components of a fiber $\XX_v$ is negative semi-definite, with
kernel $(1/m_v)\XX_v$.  From this one deduces that $L^2\NS(\XX)$ is
the group generated by the irreducible components of fibers, with
relations $\XX_v=\XX_{v'}$ for any two closed points $v$ and $v'$.
Note that if for some $v\neq v'$ we have $m=\gcd(m_v,m_{v'})>1$, then
$L^2\NS(\XX)$ has non-trivial torsion: $m((1/m)\XX_v-(1/m)\XX_{v'})=0$
in $\NS(\XX)$.

Summarizing the above, we have that $\NS(\XX)/L^1\NS(\XX)$ is infinite
cyclic, and $L^2\NS(\XX)$ is finitely generated of rank
$1+\sum_v(f_v-1)$, where $f_v$ is the number of irreducible components
of $\XX_v$.  Also, $L^2\NS(\XX)$ is torsion free if and only if
$\gcd(m_v,m_{v'})=1$ for all $v\neq v'$.

The interesting part of $\NS(\XX)$, namely $L^1\NS(\XX)/L^2\NS(\XX)$,
is the subject of the Shioda-Tate theorem.

\section{Shioda-Tate theorem}\label{s:S-T}
Write $\divr^0(X)$ for the group of divisors on $X$ of degree 0.
Restriction to the generic fiber gives a homomorphism
$L^1\divr(\XX)\to\divr^0(X)$ which descends to a homomorphism 
$L^1\Pic(\XX)\to J_X(K)=\Pic^0_{X/K}(K)$.  The Shioda-Tate theorem
uses this map to describe $L^1\NS(\XX)/L^2\NS(\XX)$.

Recall that $(B,\tau)$ is the $K/k$-trace of $J_X$.

\begin{prop}[Shioda-Tate]\label{prop:S-T}
The map above induces a homomorphism
$$\frac{L^1\NS(\XX)}{L^2\NS(\XX)}\to \MW(J_X)=\frac{J_X(K)}{\tau
  B(k)}$$
with finite kernel and cokernel.  In particular the two sides have the
same rank as finitely generated abelian groups.  This homomorphism is
an isomorphism if $X$ has a $K$-rational point and $k$ is either finite
or algebraically closed
\end{prop}

Taking into account what we know about $\NS(\XX)/L^1\NS(\XX)$ and
$L^2\NS(\XX)$, we have a formula relating the ranks of $\NS(\XX)$ and
$\MW(J_X)$.

\begin{cor}\label{cor:STcor}
$$\rk\NS(\XX)=\rk \MW(J_X) +2+\sum_v(f_v-1)$$
\end{cor}

Various versions of Proposition~\ref{prop:S-T} appear in the
literature, notably in \cite{Tate66b}, \cite{Gordon79},
\cite{Shioda99}, and \cite{HindryPacheco05}, and as Shioda notes
\cite[p.~359]{Shioda99}, it was surely known to the ancients.

\begin{proof}[Proof of Proposition~\ref{prop:S-T}]
  Specialized to $X/K$, the exact sequence of low degree terms
  \eqref{eq:low-degree} gives an exact sequence
$$0\to\Pic(X)\to\Pic_{X/K}(K)\to\Br(K).$$
If $X$ has a $K$-rational point, $\Pic(X)\to\Pic_{X/K}(K)$ is an
isomorphism.  In any case, $X$ has a rational point over a finite
extension $K'$ of $K$ of degree $\delta$ (the index of $X$), and over
$K'$ the coboundary in the analogous sequence
$$\Pic_{X/K'}(K')\to\Br(K')$$ 
is zero.  This implies that the cokernel of $\Pic(X)\to\Pic_{X/K}(K)$
maps to the kernel of $\Br(K)\to\Br(K')$ and therefore lies in the
$\delta$-torsion subgroup of $\Br(K)$ \cite[p.~157]{SerreLF}.  The
upshot is that the cokernel of $\Pic(X)\to\Pic_{X/K}(K)$ has exponent
dividing $\delta$.

A simple geometric argument as in Shioda \cite[p.~363]{Shioda99} shows
that the kernel of the homomorphism $L^1\Pic(\XX)\to\Pic^0(X)$ is
exactly $L^2\Pic(\XX)$, so we have an isomorphism
$$\frac{L^1\Pic(\XX)}{L^2\Pic(\XX)}\cong\Pic^0(X).$$

Now consider the composed homomorphism
$$\Pic^0(\XX)=\Pic^0_{\XX/k}(k)\to\Pic^0(X)\to\Pic^0_{X/K}(K).$$
There is an underlying homomorphism of algebraic groups 
$$\Pic^0_{\XX/k}\times_kK\to\Pic^0_{X/K}$$
inducing the homomorphism above on points.
By the definition of the $K/k$-trace, this morphism must factor
through $B$, i.e., we have a morphism of algebraic groups over $k$:
$$\Pic^0_{\XX}\to B.$$

We are going to argue that this last morphism is surjective.  To do so
first note that a similar discussion applies over $\kbar$ and yields
the following diagram:
\begin{equation*}
\xymatrix{
\frac{ L^1\Pic(\XX\times\kbar)}{L^2\Pic(\XX\times\kbar)}\ar[r]&
\Pic^0(X\times K\kbar)\ar[r]&\Pic(X\times K\kbar)\ar[r]&\Pic_{X/K\kbar}(K\kbar)\\
\Pic^0(\XX\times\kbar)\ar[u]\ar[rrr]&&&B(\kbar)\ar[u]}
\end{equation*}
Now the cokernel of the left vertical map is a subquotient of
$\NS(\XX)$, so finitely generated, and the cokernels of the horizontal
maps across the top are trivial, $\Z$, and of finite exponent
respectively.  On the other hand, $B(\kbar)$ is a divisible group.
This implies that the image of $B(\kbar)$ in $\Pic_{X/K\kbar}(K\kbar)$
is equal to the image of $\Pic^0(\XX\times\kbar)$ in that same
group.  Since the kernel of $B(\kbar)\to\Pic_{X/K\kbar}(K\kbar)$ is
finite, this is turn implies that the morphism $\Pic^0(\XX)\to B$ is
surjective.

An alternative proof of the surjectivity is to use the exact sequence
\eqref{eq:H1s}.  The middle term is $V_\ell\Pic^0(\XX\times\kbar)$
whereas the right hand term is $V_\ell B(\kbar)$.  Thus the morphism
$\Pic^0_{\XX}\to B$ is surjective.

Next we argue that the map of $k$-points $\Pic(\XX)=\Pic_{\XX/k}(k)\to
B(k)$ has finite cokernel.  More generally, if $\phi:A\to A'$ is a
surjective morphism of abelian varieties over a field $k$, then we
claim that the map of points $A(k)\to A'(k)$ has finite cokernel.  If
$\phi$ is an isogeny, then considering the dual isogeny $\phi^\vee$
and the composition $\phi\phi^\vee$ shows that the cokernel is killed
by $\deg\phi$, so is finite.  For a general surjection, if $A''\subset
A$ is a complement (up to a finite group) of $\ker\phi$, then by the
above $A''(k)\to A'(k)$ has finite cokernel, and {\it a fortiori\/} so
does $A(k)\to A'(k)$.  (Thanks to Marc Hindry for suggesting this
argument.)

When $k$ is algebraically closed, $\Pic^0_{\XX}(k)\to B(k)$ is
obviously surjective.  It is also surjective when $k$ is finite and
$X$ has a $K$-rational point.  This follows from Lang's theorem
because according to Proposition~\ref{prop:explicitB} below, the
kernel of $\Pic^0_{\XX}\to B$ is an abelian variety.

Summing up, the map
$$\frac{L^1\Pic(\XX)}{L^2\Pic(\XX)}\to J_X(K)$$
has cokernel of finite exponent, and 
$$\Pic^0(\XX)\to B(k)$$
has finite cokernel.  Thus the map
$$\frac{L^1\NS(\XX)}{L^2\NS(\XX)}=
\frac{L^1\Pic(\XX)}{L^2\Pic(\XX)+\Pic^0(\XX)} \to
\MW(J_X)=\frac{J_X(K)}{\tau B(k)}$$ 
has finite kernel and its cokernel has finite exponent.  Since the
target group is finitely generated (Lang-N\'eron), the cokernel must
be finite.  When $X$ has a rational point the first map displayed
above is surjective, and when $k$ is finite or algebraically closed,
the second map is surjective.  Thus under both of these hypotheses,
the third map is an isomorphism.

This completes the proof of the theorem.
\end{proof}

 \begin{rem} 
   It is also possible to deduce the theorem in the case when
   $k$ is finite from the case when $k$ is algebraically closed by
   taking Galois invariants.
 \end{rem}
 
It will be convenient to have an explicit description of $B$.

\begin{prop}\label{prop:explicitB}
  If $X$ has a $K$-rational point and $k$ is perfect, then we have an
  exact sequence
$$0\to\Pic^0_{\CC/k}\to\Picvar_{\XX/k}\to B\to 0.$$
\end{prop}

\begin{proof}[Sketch of proof]
Obviously there are morphisms $\Pic^0_{\CC/k}\to\Picvar_{\XX/k}\to
B$.  The first is injective because $\pi:\XX\to\CC$ has a section. and the
second was seen to be surjective in the proof of the Shioda-Tate
theorem.  It is also clear that the composed map is zero.

Again because $\pi$ has a section, the argument in
Section~\ref{s:Leray} shows that the integral $\ell$-adic Leray
spectral sequence degenerates and we have exact sequences
$$0\to H^1(\CCbar,\Zl(1))\to H^1(\XXbar,\Zl(1))\to \
H^0(\CCbar,R^1\pi_*\Zl(1))\to0$$ for all $\ell$.  These cohomology
groups can be identified with the Tate modules of $\Pic^0_{\CC/k}$,
$\Picvar_{\XX/k}$, and $B$ respectively.  (For $\ell=p$ we should use
flat cohomology.)  This shows that $\Picvar_{\XX/k}/\Pic^0_{\CC/k}\to
B$ is purely inseparable.  But we also have
$$0\to H^1(\CC,\OO_\CC)\to H^1(\XX,\OO_\XX)\to 
H^0(\CC,R^1\pi_*\OO_\XX)\to0$$ 
and (using the existence of a section and thus cohomological flatness)
these three groups can be identified with the tangent spaces of
$\Pic^0_{\CC/k}$, $\Picvar_{\XX/k}$, and $B$ respectively.  This shows
that $\Picvar_{\XX/k}/\Pic^0_{\CC/k}\to B$ is separable.  Thus it must
be an isomorphism.
\end{proof}

\begin{rem}
  If we no longer assume $k$ is perfect, but continue to assume $X$
  has a rational point, then the conclusion of the proposition still
  holds.  Indeed, we have a morphism
  $\Picvar_{\XX/k}/\Pic^0_{\CC/k}\to B$ which becomes an isomorphism
  over $\overline{k}$.  (Here we use that $K/k$ is regular, so that
  formation of $B$ commutes with extension of $k$ \cite[6.8]{Conrad06}
  and similarly for the Picard varieties.)  It must therefore have
  already been an isomorphism over $k$.  In general, the morphism
  $\Picvar_{\XX/k}/\Pic^0_{\CC/k}\to B$ may be purely inseparable.
\end{rem}

\section{Heights}
We use the Shioda-Tate theorem to define a non-degenerate bilinear
pairing on $\MW(J_X)$ which is closely related to the N\'eron-Tate
canonical height when $k$ is finite.

To simplify the notation we write
$$L^1=(L^1\NS(\XX))\tensor\Q\quad\text{and}\quad
L^2=(L^2\NS(\XX))\tensor\Q$$ 
so that $\MW(J_X)\tensor\Q\cong L^1/L^2$. 

It is well known that the intersection form on $\NS(\XX)\tensor\Q$ is
non-degenerate (cf.~\cite[p.~29]{Ulmer11}).  Since $L^1$ has
codimension 1 in $\NS(\XX)\tensor\Q$, the kernel of the intersection
form restricted to $L^1$ has dimension at most 1, and it is easy to
see that it is in fact 1-dimensional, generated by $F$, the class of a
fiber of $\pi$.  Also, our discussion of the structure of the fibers
of $\pi$ shows that the kernel of the intersection form restricted to
$L^2$ is also 1-dimensional, generated by $F$.

It follows that for each class in $D\in L^1/L^2$, there is a
representative $\tilde D\in L^1$ such that $\tilde D$ is orthogonal to
all of $L^2$; moreover, $\tilde D$ is determined by $D$ up to the
addition of a multiple of $F$.  We thus have a homomorphism
$\phi:L^1/L^2\to L^1/\Q F$ defined by $\phi(D)=\tilde D$.  We define a
bilinear pairing on $\MW(J_X)\tensor\Q=L^1/L^2$ by the rule
$$(D,D')=-\phi(D).\phi(D')$$
where the dot signifies the intersection pairing on $\XX$, extended by
$\Q$-linearity.  The right hand side is well defined because $F\in L^2$. 

\begin{prop}
  The formula above defines a symmetric, $\Q$-valued
  bilinear form on $\MW(J_X)\tensor\Q$ which is positive definite on
  $\MW(J_X)\tensor\R$.  If $k$ is a finite field of cardinality $q$,
  then $(\cdot,\cdot)\log q$ is the N\'eron-Tate canonical height on
  $\MW(J_X) \cong J_X(K)\tensor\Q$
\end{prop}

\begin{proof}[Sketch of proof]
It is clear that the pairing is bilinear, symmetric, and $\Q$-valued.  

To see that it is positive definite on $\MW(J_X)\tensor\Q$, we use the
Hodge index theorem.  Given $D\in L^1/L^2$, choose an irreducible
multisection $P$ and a representative $\tilde\phi(D)\in L^1$ for
$\phi(D)$.  The intersection number $P.D$ is in $\Q$ and replacing $D$
with $nD$ for suitable $n$ we can assume it is in $\Z$.  Then adding a
multiple of $F$ to $\tilde\phi(D)$ and calling the result again
$\tilde\phi(D)$, we get a new representative of $\phi(D)$ such that
$P.\tilde\phi(D)=0$.  On the other hand $(P+mF)^2>0$ if $m$ is large
enough and $(P+mF).\tilde\phi(D)=0$.  The Hodge index theorem (e.g.,
\cite[2.4]{BadescuAS}) then implies that $\tilde\phi(D)^2<0$ and so
$(D,D)>0$.

Since the pairing is positive definite and $\Q$-valued on
$\MW(J_X)\tensor\Q$, it is also positive definite on
$\MW(J_X)\tensor\R$.

The connection with N\'eron-Tate canonical height can be seen by using
the local canonical heights.  These were constructed purely on the
relative curve $\XX\to\CC$ (as opposed to on the N\'eron model
$\JJ_X\to\CC$) by Gross in \cite{Gross86}.  An inspection of his
construction shows that summing the local heights over all places of
$K$ gives exactly our definition above times $\log q$.
\end{proof}

\begin{exer}
  When $X$ is an elliptic curve, give a global proof of the connection
  between the Shioda-Tate height and the canonical height as follows:
  Check from the definition that $\log q$ times the Shioda-Tate height
  differs from the naive height associated to the origin of $X$ (as a
  divisor on $X$) by a bounded amount.  Since the Shioda-Tate height
  is already bilinear, it must thus be the canonical height as
  constructed by Tate.
\end{exer}

\chapter{Cohomology and cycles}\label{ch:cohom}

\section{Tate's conjecture for $\XX$}
Throughout this section, $k$ is a finite field 
of characteristic $p$
and $\XX$ is a surface
smooth and projective over $k$.  We write $\XXbar$ for
$\XX\times_k\kbar$.  Our goal is to discuss Tate's first conjecture on
cycles and cohomology classes for $\XX$.  Since this was already
discussed in \cite[Lecture~2]{Ulmer11}, we will be brief.

\subsection{Basic exact sequences}\label{ss:seqs}
Let $\Br(\XX)=H^2(\XX,\OO_X^\times)=H^2(\XX,\G_m)$ be the
(cohomological) Brauer group of $\XX$.  Here we use the \'etale or
flat topologies.  It is known that $\Br(\XX)$ is a torsion group
and it is conjectured to be finite in our situation.

For any positive integer $n$, consider the Kummer sequence on
$\XXbar$:
$$0\to\mu_{n}\to\G_m\labeledto{n}\G_m\to0.$$
Here if $p|n$ the sequence is exact in the flat topology, but not in the
\'etale topology.

Let $\ell$ be a prime (with $\ell=p$ allowed).
Taking cohomology and noting that
$\Pic(\XXbar)/\ell^n=\NS(\XXbar)/\ell^n$ (because $\Pic^0(\XXbar)$ is
divisible), we find an exact sequence
$$0\to\NS(\XXbar)/\ell^n\to
H^2(\XXbar,\mu_{\ell^n})\to\Br(\XXbar)[\ell^n]\to0.$$

Taking the inverse limit over $n$ yields an exact sequence
$$0\to\NS(\XXbar)\tensor\Zl\to H^2(\XXbar,\Zl(1))\to T_\ell\Br(\XXbar)\to0.$$

Now $H^1(G_k,\NS(\XXbar)\tensor\Zl)$ is a finite group.  On the
other hand, if $H$ is a group with finite $\ell$-torsion, $T_\ell H$ is
finitely generated and torsion-free over $\Zl$.  This applies in
particular to $T_\ell\Br(\XXbar)$.  Thus taking
$G_k$ invariants, we get an exact sequence
$$0\to\left(\NS(\XXbar)\tensor\Zl\right)^{G_k}\to H^2(\XXbar,\Zl(1))^{G_k}\to
\left(T_\ell\Br(\XXbar)\right)^{G_k}\to 0.$$

Next note that since $k$ is finite, $H^1(G_k,\Pic^0(\XXbar))=0$ and so 
$$\left(\NS(\XXbar)\tensor\Zl\right)^{G_k}=\NS(\XX)\tensor\Zl$$
since by definition $\NS(\XX)$ is the image of $\Pic(\XX)$ in
$\NS(\XXbar)$.

Now using the Hochschild-Serre spectral sequence, we have a
homomorphism $H^2(\XX,\Zl(1))\to H^2(\XXbar,\Zl(1))^{G_k}$ which is
surjective with finite kernel.  Since $T_\ell\Br(\XX)$ is
torsion-free, the commutative diagram
$$\xymatrix{
&\NS(\XX)\tensor\Zl\ar[d]\ar@{=}[r]&
\left(\NS(\XXbar)\tensor\Zl\right)^{G_k}\ar[d]&\\
(\textrm{finite})\ar[r]& H^2(\XX,\Zl(1))\ar[r]\ar[d]& 
H^2(\XXbar,\Zl(1))^{G_k}\ar[d]\ar[r]&0\\
&T_\ell\Br(\XX)\ar[d]\ar[r]&T_\ell\Br(\XXbar)^{G_k}\ar[d]&\\
&0&0&}$$ 
induces an isomorphism $T_\ell\Br(\XX)\isoto T_\ell\Br(\XXbar)^{G_k}$.

Putting everything together, we get an exact sequence
\begin{equation}\label{eq:cycleclass}
0\to\NS(\XX)\tensor\Zl\to H^2(\XXbar,\Zl(1))^{G_k}\to
T_\ell\Br(\XX)\to 0
\end{equation}

\subsection{First Tate conjecture}
Tate's first conjecture for a prime $\ell$, which we denote
$T_1(\XX,\ell)$, says that the cycle class map induces an isomorphism
$$\NS(\XX)\tensor\Ql\cong H^2(\XXbar,\Ql(1))^{G_k}.$$

The exact sequence \eqref{eq:cycleclass} at the end of the previous
subsection leads immediately to several equivalent forms of this
conjecture.

\begin{prop}\label{prop:T1<=>Br}
The following are equivalent:
\begin{enumerate}
\item $T_1(\XX,\ell)$
\item The cycle class map induces an isomorphism
$$\NS(\XX)\tensor\Zl\cong H^2(\XXbar,\Zl(1))^{G_k}$$
\item $T_\ell \Br(\XX)=0$
\item The $\ell$ primary component of $\Br(\XX)$ is finite.
\end{enumerate}
\end{prop}

\begin{proof}
Tensoring Equation~\eqref{eq:cycleclass} with $\Ql$ shows that
$T_1(\XX,\ell)$ is equivalent to
$V_\ell\Br(\XX):=T_\ell\Br(\XX)\tensor\Ql=0$ and therefore to
$T_\ell\Br(\XX)=0$ which is in turn equivalent to the injection
$$\NS(\XX)\tensor\Zl\to H^2(\XXbar,\Zl(1))$$
being an isomorphism.  This establishes the equivalence of 1-3.  Since
$\Br(\XX)$ is a torsion group and $\Br(\XX)[\ell]$ is finite, the
$\ell$-primary component of $\Br(\XX)$ is finite if and only if
$T_\ell Br(\XX)$ is finite.  This establishes the equivalence of 3 and
4.
\end{proof}

We will see below that $T_1(\XX,\ell)$ is equivalent to $T_1(\XX,\ell')$
for any two primes $\ell$ and $\ell'$.

We note also that the flat cohomology group 
$H^2(\XXbar,\Qp(1))$ 
is isomorphic to the crystalline cohomology group
$(H^2(\XXbar/W(\kbar))\tensor_{\Zp}\Qp)^{\Fr=p}$ where the exponent
indicates the subspace upon which the absolute Frobenius $\Fr$ acts by
multiplication by $p$.  One can thus reformulate the $\ell=p$ part
of the Tate conjecture in terms of crystalline cohomology.

\section{Selmer group conjecture for $J_X$}
We assume throughout this section that the ground field $k$ is finite.

By the Lang-N\'eron theorem discussed in Section~\ref{s:LangNeron}
above, the Mordell-Weil group $J(K)/\tau B(k)$ is finitely generated,
and since $B(k)$ is obviously finite, we have that $J_X(K)$ is also
finitely generated.  In this section we quickly review the standard
mechanism to bound (and conjecturally capture) the rank of this group
via descent.

\subsection{Selmer and Tate-Shafarevich groups}\label{ss:Sel-Sha}
For each positive integer $n$, consider the exact sequence of sheaves
for the flat topology on $\spec K$:
$$0\to J_X[n]\to J_X\labeledto{n} J_X\to0$$
where $J_X[n]$ denotes the kernel of
multiplication by $n$ on $J_X$.  Taking cohomology yields
$$0\to J_X(K)/nJ_X(K)\to H^1(K,J_X[n])\to H^1(K,J_X)[n]\to0$$

Similarly, for every completion $K_v$ of $K$ we have a similar
sequence and restriction maps.
\begin{equation*}
\xymatrix{
0\ar[r]& J_X(K)/nJ_X(K)\ar[r]\ar[d]& H^1(K,J_X[n])\ar[r]\ar[d]&
H^1(K,J_X)[n]\ar[r]\ar[d]&0\\
0\ar[r]& J_X(K_v)/nJ_X(K_v)\ar[r]& H^1(K_v,J_X[n])\ar[r]&
H^1(K_v,J_X)[n]\ar[r]&0}
\end{equation*}

We define $\Sel(J_X,n)$ to be
$$\ker\left(H^1(K,J_X[n])\to\prod_v H^1(K_v,J_X)\right)$$
where the product is over the places of $K$.  Also, for any prime
$\ell$, we define $\Sel(J_X,\Zl)=\varprojlim_n \Sel(J_X,\ell^n)$.
We define the Tate-Shafarevich group by
$$\sha(J_X)=\ker\left(H^1(K,J_X)\to \prod_v H^1(K_v,J_X)\right).$$

We have exact sequences
$$0\to J_X(K)/nJ_X(K)\to\Sel(J_X,n)\to\sha(J_X)[n]\to0$$
for each $n$.  All the groups appearing here are finite.  (For $n$
prime to $p$, the classical proof of finiteness of the Selmer
group---reducing it to the finiteness of the class group and finite
generation of the unit group of Dedekind domains with fraction field
$K$---works in our context.  The $p$ part was first proven by Milne
\cite{Milne70}.)  So taking the inverse limit over powers of a prime
$\ell$ yields an exact sequence
\begin{equation}\label{eq:descent}
0\to J_X(K)\tensor\Zl\to\Sel(J_X,\Zl)\to T_\ell\sha(J_X)\to 0
\end{equation}

\subsection{Tate-Shafarevich conjecture}
Tate and Shafarevich conjectured (independently) that $\sha(J_X)$ is
a finite group.  We write $TS(J_X)$ for this conjecture, and
$TS(J_X,\ell)$ for the {\it a priori\/} weaker conjecture that the
$\ell$-primary part of $\sha(J_X)$ is finite.

For each prime $\ell$, we refer to the statement ``the homomorphism
$$J_X(K)\tensor\Zl\to\Sel(J_X/K,\Zl)$$
is an isomorphism'' as the {\it Selmer group conjecture\/}
$S(J_X,\ell)$.  (This is perhaps non-standard and is meant to suggest
that the Selmer group captures rational points.)

The following is obvious from the exact sequence \eqref{eq:descent} above:

\begin{prop}\label{prop:Sel<=>TS}
 For each prime $\ell$, the Selmer group conjecture
  $S(J_X,\ell)$ holds if and only if the Tate-Shafarevich conjecture
  $TS(J_X,\ell)$ holds.
\end{prop}

Note that if the $\ell$ primary component of $\sha(J_X)$ is finite,
then knowing $\Sel(J_X,\ell^{n})$ for sufficiently large $n$
determines the rank of $J_X(K)$.  This observation and some input from
$L$-functions leads (conjecturally) to an effective algorithm for
computing generators of $J_X(K)$, see \cite{Manin71}.

\section{Comparison of cohomology groups}
There is an obvious parallel between the conjectures $T_1(\XX,\ell)$
and finiteness of $\Br(\XX)[\ell^\infty]$ on the one hand, and
$S(J_X,\ell)$ and $TS(J_X,\ell)$ on the other.  The Shioda-Tate
isomorphism gives a precise connection between the groups of cycles
and points involved.  In this section, we give 
connections between the cohomology groups involved.  We restrict to
the simplest cases here---they are already complicated enough.  In
Section~\ref{s:analytic-comparison} we will give a comparison for all
$\ell$ and the most general hypotheses on $\XX$ and $X$ using analytic
methods.

Thus for the rest of this section, the following hypotheses are in
force (in addition to the standing hypotheses): $k$ is finite of
characteristic $p$, $\ell$ is a prime $\neq p$, and $X$ has a
$K$-rational point.  This implies that $\XX\to\CC$ has a section, and
so the multiplicities $m_v$ of the fibers of $\pi$ are
all equal to 1 and $\pi$ is cohomologically flat.

\subsection{Comparison of $\Br(\XX)$ and $\sha(J_X)$}
These groups are closely related---as we will see later, under the
standing hypotheses and assuming $k$ is finite, they differ by a
finite group.  Assuming also that $X$ has a $K$-rational point, even
more is true.

\begin{prop}[Grothendieck]
  In addition to the standing hypotheses, assume that $k$ is finite
  and $X$ has a $K$-rational point.  Then we
  have a canonical isomorphism $\Br(\XX)\cong\sha(J_X)$.
\end{prop}

This is proven in detail in \cite[\S4]{Grothendieck68} by an argument
``assez long et technique.''  We sketch the main points in the rest of
this subsection.

First of all, since $\pi$ is cohomologically flat, we have
$\pi_*\G_m=\G_m$.  Also, we have a vanishing/cohomological dimension
result of Artin: $R^q\pi_*\G_m=0$ for $q>1$.  Thus the Leray spectral
sequence for $\pi$ and $\G_m$ gives a long exact sequence
$$H^2(\CC,\G_m)\to H^2(\XX,\G_m)\to H^1(\CC,R^1\pi_*\G_m)\to
H^3(\CC,\G_m)\to H^3(\XX,\G_m).$$
Now $H^2(\CC,\G_m)=\Br(\CC)=0$ since $\CC$ is a smooth complete curve
over a finite field.  Also, since $\pi$ has a section,
$H^3(\CC,\G_m)\to H^3(\XX,\G_m)$ is injective.  Thus we have an isomorphism
$$\Br(\XX)=H^2(\XX,\G_m)\cong H^1(\CC,R^1\pi_*\G_m).$$

Now write $\FF=R^1\pi_*\G_m$ and consider the inclusion $\eta=\spec
K\into\CC$.  Let $\GG=j_*j^*\FF$ so that we have a canonical morphism
$\FF\to\GG$.  If $\xbar$ is a geometric point over a closed point of $\CC$, we write
$\tilde\XX_\xbar$, $\tilde\eta_\xbar$, and $\tilde\CC_\xbar$ for the
strict henselizations at $\xbar$.  Then the stalk
of $\FF\to\GG$ at $\xbar$ is
$$\Pic(\tilde\XX_\xbar/\tilde\CC_\xbar)
\to\Pic(\tilde \XX_{\tilde\eta_\xbar}/\tilde\eta_{\xbar}).$$
Since the Brauer group of $\tilde\eta_\xbar$ vanishes, the displayed
map is surjective (cf.~the exact sequence~\eqref{eq:low-degree}).  The
kernel is zero at all $\xbar$ where the fiber of $\pi$ is reduced and
irreducible.  Thus we have an exact sequence
$$0\to\KK\to\FF\to\GG\to0$$ 
where  $\KK$ is a skyscraper sheaf supported on the points of $\CC$
where the fibers of $\pi$ are reducible.  (See
\cite[pp. 115--118]{Grothendieck68} for a more complete description of
this sheaf and its cohomology.)  Under our hypotheses ($k$ finite and
$X$ with a $K$-rational point), every fiber of $\pi$ has a component
of multiplicity 1, and this is enough to ensure that
$H^1(\CC,\KK)=0$.  

Now write $\XX_x$ and $\CC_x$ for the ordinary henselizations of $\XX$
and $\CC$ at a closed point $x$ of $\CC$.  Then we have a diagram
\begin{equation*}
\xymatrix{
0\ar[r]&H^1(\CC,\FF)\ar[r]\ar[d]&H^1(\CC,\GG)\ar[r]\ar[d]
&\prod_xH^2(x,\KK_x)\ar@{=}[d]\\
0\ar[r]&\prod_xH^1(\CC_x,\FF)\ar[r]&\prod_xH^1(\CC_x,\GG)\ar[r]
&\prod_xH^2(x,\KK_x).}
\end{equation*}
Lang's theorem implies that $H^1(\CC_x,\FF)=0$ for all $x$.  Thus
$\Br(\XX)=H^1(\CC,\FF)$ is identified with the subgroup of
$H^1(\CC,\GG)$ consisting of elements which go to zero in
$H^1(\CC_x,\GG)$ for all $x$.  Call this group $\sha(\CC,\GG)$.

Finally, the Leray spectral sequence for $\eta\into\CC$ and the
definition of $\GG$ leads to an exact sequence
$$0\to\sha(\CC,\GG)\to
H^1(\eta,j^*\FF)\to\coprod_xH^1(\kappa(\CC_x),j^*\FF)$$
where $\kappa(\CC_x)$ is the field of fractions of the henselization of
$\CC$ at $x$.  This identifies $\sha(\CC,\GG)$ with $\sha(J_X)$ as
defined in Subsection~\ref{ss:Sel-Sha}.

This completes our sketch of the Grothendieck's theorem. The reader is
encouraged to consult \cite{Grothendieck68} for a much more general
discussion delivered in the inimitable style of the master.

One consequence of the proposition is that (under the hypotheses
there) the conjectures on the finiteness of $\Br(\XX)$ and $\sha(J_X)$
are equivalent, and thus so are $T_1(\XX,\ell)$ and $S(J_X,\ell)$ for all
$\ell$.  We will see below that this holds even without the supplementary
hypothesis on $X$.

\subsection{Comparison of the Selmer group and
  $H^1(\CCbar,R^1\pi_*\Zl(1))^{G_k}$}\label{s:Sel-H1R1} 
The arguments in Section~\ref{s:Leray} show that under our hypotheses
the Leray spectral sequence degenerates at the integral level as well
and so $H^2(\XXbar,\Zl(1))$ has a filtration whose graded pieces are
the groups $H^i(\CCbar,R^j\pi_*\Zl(1))$ where $i+j=2$.

As we will see below, from the point of view of the Tate conjecture,
the interesting part of $H^2(\XXbar,\Zl(1))$ is
$H^1(\CCbar,R^1\pi_*\Zl(1))$.  It turns out that its $G_k$-invariant
part is closely related to the $\ell$-adic Selmer group of
$J_X$---they differ by a finite group.  

Before making this more precise, we make the groups
$H^i(\CCbar,\Zl(1))$ more explicit.  Since the fibers of $\pi$ are
connected, $\pi_*\Zl(1)=\Zl(1)$ and so $H^2(\CCbar,\pi_*\Zl(1))=\Zl$.
It is easy to see that under the map $H^2(\XXbar,\Zl(1))\to
H^2(\CCbar,\pi_*\Zl(1))$, the cycle class of a section maps to a
generator of $H^2(\CCbar,\pi_*\Zl(1))$.

The stalk of $R^2\pi_*\Zl(1)$ at a geometric point $\xbar$ over a
closed point $x$ of $\CC$ is $\Zl^{f_{\xbar}}$ where $f_{\xbar}$ is
the number of irreducible components in the geometric fiber.  The
action of $G_k$ permutes the factors as it permutes the components.
If we write the fiber as $\sum_i m_{\xbar,i}\XX_{\xbar,i}$ as in
Section~\ref{s:fibers}, then the specialization map
$$\Zl^{f_{\xbar}}=(R^2\pi_*\Zl(1))_\xbar\to
(R^2\pi_*\Zl(1))_{\etabar}=\Zl$$ 
is $(c_i)\mapsto \sum_i m_{\xbar,i}c_i$.  It follows that
$H^0(\CCbar,R^2\pi_*\Zl(1))^{G_k}$ has rank $1+\sum_v(f_v-1)$
where $f_v$ is the number of irreducible components of the fiber
$\XX_v$ (as a scheme over the residue field $k_v$).  Also, the cycle
classes of components of fibers lie in
$$H^0(\CCbar,R^2\pi_*\Zl(1))^{G_k}\subset H^2(\XXbar,\Zl(1))^{G_k}$$
and span this subgroup.


Thus we see that the cycle class map induces a well defined
homomorphism 
$$(L^1\NS(\XX)/L^2\NS(\XX))\tensor\Zl
\to H^1(\CCbar,R^1\pi_*\Zl(1))^{G_k}$$
which is surjective if and only if $T_1(\XX,\ell)$ holds.  Since the
source of this map is closely related to $J_X(K)$---it is exactly
$(J_X(K)/\tau B(k))\tensor\Zl$ by the Shioda-Tate theorem and $B(k)$ is
finite---we must have a close connection between the groups
$H^1(\CCbar,R^1\pi_*\Zl(1))^{G_k}$ and $\Sel(J_X,\Zl)$.  We
will prove directly that they differ by a finite group.

We should say sketch a proof, since a complete treatment requires more
details on bad reduction and N\'eron models than we have at our
disposal, but the main ideas should be clear.
 
To state the result, let $K'=K\kbar$ and define another Selmer group
as follows:
$$\Sel(J_x/K',\Zl)=
\varprojlim_n\ker\left(H^1(K',J_X[\ell^n])\to\prod_v H^1(K'_v,J_X)\right)$$
where the product is over the places of $K'$.  

\begin{prop}\label{prop:Sel-H1R1}
  We have homomorphisms $\Sel(J_X,\Zl)\to\Sel(J_X/K',\Zl)^{G_k}$ and
  $H^1(\CCbar,R^1\pi_*\Zl(1))^{G_k}\to\Sel(J_X/K',\Zl)^{G_k}$ with
    finite kernels and cokernels.
\end{prop}

\begin{proof}[Sketch of proof of Proposition~\ref{prop:Sel-H1R1}]
Inflation-restriction for $G_{K'}\subset G_K$ gives sequences 
\begin{multline*}
0\to H^1(G_k,J_X[\ell^n](K'))\to H^1(G_K,J_X[\ell^n])\to
H^1(G_{K'},J[\ell^n])\\
\to H^2(G_k,J_X[\ell^n](K'))
\end{multline*}
Now $J[\ell^n](K')$ is an extension of (finite) by $B[\ell^n](\kbar)$
where (finite) is bounded independently of $n$, and $B$ is the
$K/k$-trace of $J_X$.  Using Lang's theorem, we have that the kernel
and cokernel of $H^1(G_K,T_\ell J)\to H^1(G_{K'},T_\ell J)$ are finite
and therefore so is the kernel of
$\Sel(J_X,\Zl)\to\Sel(J_X/K',\Zl)^{G_k}$.

To control the cokernel, consider the commutative diagram with exact
rows and columns.
\begin{equation*}
\xymatrix{
&0\ar[d]&0\ar[d]\\
&\Sel(J_X,\ell^n)\ar[r]\ar[d]&\Sel(J_X/K',\ell^n)^{G_k}\ar[d]\\
&H^1(G_K,J_X[\ell^n])\ar[r]\ar[d]&
H^1(G_{K'},J[\ell^n])^{G_k}\ar[d]\\
\prod_v H^1(G_{k_v},J_X(\kbar K_v))\ar[r]
&\prod_v H^1(G_{K_v},J)\ar[r]
&\prod_v\left(\prod_{\vbar|v} H^1(G_{K'_\vbar},J)\right)^{G_k}
}
\end{equation*}

The columns define the Selmer groups and the bottom row is a product
of inflation-restriction sequences.  By \cite[I.3.8]{MilneADT}, the
group on the left is zero for all but finitely many $v$ and is finite
at all $v$.  A diagram chase then shows that
$\Sel(J_X,\Zl)\to\Sel(J_X/K',\Zl)^{G_k}$ has finite cokernel.

Now let $\FF=j_*j^*\Zl(1)$ where $j:\eta\into\CC$ is the inclusion of
the generic point.  By Corollary~\ref{cor:R1H1-cohom}, we may replace
$H^1(\CCbar,R^1\pi_*\Zl(1))$ with $H^1(\CCbar,\FF)$.  The fact
that $\FF$ is a middle extension sheaf gives a $\sha$-like description
of $H^1(\CCbar,\FF)$:  We have an exact sequence
$$0\to H^1(\CCbar,\FF)\to H^1(K',T_\ell J_X)\to
\prod_{\vbar} H^1(K'_{\vbar},T_\ell J_X).$$

On the other hand, $\Sel(J_X/K',\Zl)$ is defined by a similar
sequence, except that instead of vanishing in $H^1(K'_{\vbar},T_\ell J_X)$, a
class in the Selmer group should land in the image of
$J_X(K'_{\vbar})\hat\tensor\Zl\to H^1(K'_{\vbar},T_\ell J_X)$.   But $J(K'_{\vbar})$ is an
extension of a finite group by an $\ell$-divisible group, and for all
but finitely many places $\vbar$ the finite group is trivial.  (This
follows from the structure of the N\'eron model of $J_X$.)  Thus we
get a inclusion $H^1(\CCbar,\FF)\subset \Sel(J_X/K',\Zl)$ with finite
cokernel.  Taking $G_k$-invariants finishes the proof.
\end{proof}

\chapter{Zeta and $L$-functions}\label{ch:zetas}

\section{Zeta functions and $T_2$ for $\XX$}
We quickly review the zeta function and the second Tate conjecture for
$\XX$.  Most of this material was covered in \cite{Ulmer11} and so
we will be very brief.

Throughout this section, $k=\Fq$ is the finite field with $q$
elements, $p$ is its characteristic, and $\XX$ is a surface satisfying
our usual hypotheses (as in Section~\ref{s:XX-intro}). 

\subsection{Zetas}
Let $N_n$ be the number of points on $\XX$ rational over $\F_{q^n}$
and form the $Z$-function
\begin{align*}
Z(\XX,T)&=\prod_{x\in\XX^0}(1-T^{\deg x})^{-1}\\
&=\exp\left(\sum_{n\ge1}N_n\frac{T^n}{n}\right)
\end{align*}
and the zeta function $\zeta(\XX,s)=Z(\XX,q^{-s})$.  

We choose an auxiliary prime $\ell\neq p$.  Then the
Grothendieck-Lefschetz trace formula
$$N_n=\sum_{i=0}^4(-1)^i\tr\left(\Fr^n_q|H^i(\XXbar,\Ql)\right)$$
leads to an expression for $Z(\XX,T)$ as a rational function of $T$:
$$Z(\XX,T)=\prod_{i=0}^4P_i(\XX,T)^{(-1)^{i+1}}$$
where
$$P_i(\XX,T)=\det\left(1-T\,\Fr_q|H^i(\XXbar,\Ql)\right)
=\prod_j(1-\alpha_{ij}T).$$

By Poincar\'e duality, there is a functional equation relating
$\zeta(\XX,1-s)$ and $\zeta(\XX,s)$, and therefore relating
$P_i(\XX,T)$ with $P_{4-i}(\XX,q/T)$.

By Deligne's theorem, the eigenvalues of Frobenius $\alpha_{ij}$ are
Weil numbers of size $q^{i/2}$.  It follows that the zeroes and poles
of $\zeta(\XX,s)$ have real parts in the sets $\{1/2, 3/2\}$ and
$\{0,1,2\}$ respectively, and that the order of pole of $\zeta(\XX,s)$
at $s=1$ is equal to the multiplicity of $q$ as an eigenvalue of
$\Fr_q$ on $H^2(\XXbar,\Ql)$.

\subsection{Tate's second conjecture}
The second main conjecture in Tate's article \cite{Tate65} relates the
$\zeta$ function to the N\'eron-Severi group.

\begin{conj}[$T_2(\XX)$]  We have
$$\rk\NS(\XX)=-\ord_{s=1}\zeta(\XX,s)$$.
\end{conj}

Recall (Equation~\eqref{eq:cycleclass}) that the cycle class map
induces an injection
$$\NS(\XX)\tensor\Zl\to H^2(\XXbar,\Zl(1))^{G_k}$$
and conjecture $T_1(\XX,\ell)$ was the statement that this map is an
isomorphism.  When $\ell\neq p$, $H^2(\XXbar,\Zl(1))^{G_k}\cong
H^2(\XXbar,\Zl)^{\Fr_q=q}$, where the latter is the subspace where the
$q$-power Frobenius acts by multiplication by $q$.  When $\ell=p$,
$H^2(\XXbar,\Zl(1))^{G_k}$ is isomorphic to
$H^2_{cris}(\XX/W(k))^{\Fr=p}$, the subgroup of crystalline cohomology
where the absolute Frobenius acts by multiplication by $p$.  Thus, in
either case the $\Z_\ell$-rank of the target of the cycle class map
(an eigenspace) is bounded above by the order of vanishing of the
$\zeta$ function (which is the multiplicity of the corresponding
eigenvalue).  This proves the first two parts of the following.

\begin{thm}[Artin-Tate, Milne]\label{thm:T1<>T2}  We have:
\begin{enumerate}
\item $\rk\NS(\XX)\le-\ord_{s=1}\zeta(\XX,s)$
\item $T_2(\XX)$ \textup{(}i.e., equality in (1)\textup{)} implies $T_1(\XX,\ell)$
  \textup{(}and thus finiteness of $\Br(\XX)[\ell^\infty]$\textup{)}
  for all $\ell$.
\item $T_1(\XX,\ell)$ for any one $\ell$ \textup{(}$\ell=p$ 
  allowed\textup{)} implies $T_2(\XX)$.
\end{enumerate}
\end{thm}

We sketch the proof of the last part in the next subsection.  The
prime-to-$p$ was discussed/sketched in \cite[Lecture 2, 10.2]{Ulmer11}
and the general case is similar, although it uses more sophisticated
cohomology.

\subsection{Artin-Tate conjecture}
We define two more invariants which enter into the Artin-Tate
conjecture below.

Recall that there is a symmetric, integral intersection pairing on the
N\'eron-Severi group which gives $\NS(\XX)/tor$ the structure of a
lattice.  We define the regulator $R(\XX)$ to be the discriminant of
this lattice.  More precisely
$$R=R(\XX)=\left|\det\left(D_i.D_j\right)_{i,j=1,\dots,\rho}\right|$$
where $D_i$ ($i=1,\dots,\rho$) is a basis of $\NS(\XX)/tor$ and the
dot denotes the intersection product.

We also define a Tamagawa-like factor $q^{-\alpha(\XX)}$ where
$$\alpha=\alpha(\XX)=\chi(\XX,\OO_\XX)-1+\dim\Picvar(\XX).$$ 

The following was inspired by the $BSD$ conjecture for the Jacobian of
$X$.

\begin{conj}[$AT(\XX)$] $T_2(\XX)$ holds, $\Br(\XX)$ is finite, and
  we have the asymptotic
$$\zeta(\XX,s)\sim
\frac{R\,|\!\Br(\XX)|\,q^{-\alpha}}{|\!\NS(\XX)_{tor}|^2}
(1-q^{1-s})^{\rho(\XX)}$$
as $s\to1$.
\end{conj}

The analysis showing that $T_1(\XX,\ell)\implies T_2(\XX)$ can be
pushed further to show that $T_2(\XX)\implies AT(\XX)$.  Indeed, we
have the following spectacular result:

\begin{thm}[Artin-Tate, Milne]\label{thm:T2=>AT}
If $T_2(\XX)$ holds, then so does $AT(\XX)$.
\end{thm}

To prove this, we must show that for all but finitely many $\ell$,
$\Br(\XX)[\ell]$ is trivial, and also relate the order of $\Br(\XX)$
and other invariants to the leading term of the zeta function.

The proofs of the third part of Theorem~\ref{thm:T1<>T2} and
Theorem~\ref{thm:T2=>AT} proceed via a careful consideration of the
following big commutative diagram:
\begin{equation*}
\xymatrix@C-15pt{
\NS(\XX)\tensor\Zhat\ar[d]_h\ar[r]^>>>>{e}&
\Hom(\NS(\XX)\tensor\Zhat,\Zl)\ar@{=}[r]&
\Hom(\NS(\XX)\tensor\Q/\Z,\Q/\Z)\\
H^2(\XXbar,T\mu)^{G_k}\ar[r]^f&H^2(\XXbar,T\mu)_{G_k}\ar[r]^>>>>{j}&
\Hom(H^2(\XX,\mu(\infty))^{G_k},\Q/\Z)\ar[u]_{g^*}}
\end{equation*}

Here $H^2(\XXbar,T\mu)$ means the inverse limit over $n$ of the flat
cohomology groups $H^2(\XXbar,\mu_n)$ and $H^2(\XX,\mu(\infty))$ is
the direct limit over $n$ of the flat cohomology groups
$H^2(\XX,\mu_n)$.  The map $e$ is induced by the intersection form,
$h$ is the cycle class map, $f$ is induced by the identity map of
$H^2(\XXbar,T\mu)$, $g^*$ is the transpose of a map
$$g:\NS(\XX)\tensor\Q/\Z\to H^2(\XX,\mu(\infty))$$
obtained by taking the direct limit over positive integers $n$ of the 
Kummer map $\NS(\XX)\tensor\Z/n\Z\to H^2(\XX,\mu_n))$, and $j$ is
induced by the Hochschild-Serre spectral sequence and Poincar\'e duality.

We say that a homomorphism $\phi:A\to B$ of abelian groups is a {\it
  quasi-isomorphism\/} if it has a finite kernel and cokernel.  In this
case, we define 
$$z(\phi)=\frac{\#\ker(\phi)}{\#\coker(\phi)}.$$  
It is easy to check that if $\phi_3=\phi_2\phi_1$ (composition) and if
two of the maps $\phi_1$, $\phi_2$, $\phi_3$ are quasi-isomorphisms,
then so is the third and we have $z(\phi_3)=z(\phi_2)z(\phi_1)$.

In the diagram above, the map $e$ is a quasi-isomorphism and $z(e)$ is
the order of the torsion subgroup of $\NS(\XX)$ divided by the
discriminant of the intersection form.  The map $j$ is also a
quasi-isomorphism and $z(j)$ is the order of the torsion subgroup of
N\'eron-Severi times a power of $q$ determined by a certain $p$-adic
cohomology group.

Now assume $T_1(\XX,\ell)$ for one $\ell$ and consider the
$\ell$-primary part of the diagram above.  The assumption
$T_1(\XX,\ell)$ implies that $\ell$ part of the maps $h$ and $g^*$ are
quasi-isomorphisms.  Thus the $\ell$ part of $f$ must also be a
quasi-isomorphism.  This means that $\Fr_q$ acts semisimply on the
generalized eigenspace of $H^2(\XXbar,\Ql(1))$ corresponding to the
eigenvalue 1 and this implies $T_2(\XX)$.

Now assume $T_2(\XX)$, so we have that $T_1(\XX,\ell)$ holds for all
$\ell$ and that $h$ is an isomorphism.  A fairly intricate analysis
leads to a calculation of $z(f)$ in terms of the zeta function of
$\XX$ and shows that $z(g^*)$ (as a ``supernatural number'') is the
order of the Brauer group.  Since all the other $z$'s in the diagram
are rational numbers, so is $z(g^*)$ and we get finiteness of
$\Br(\XX)$.  The product formula $z(e)=z(h)z(f)z(j)z(g^*)$ leads,
after some delicate work in $p$-adic cohomology, to the leading
coefficient formula in the conjecture $AT(\XX)$.  We refer to Milne's
paper \cite{Milne75} for the details.

\subsection{The case $p=2$ and 
de\thinspace Rham-Witt cohomology}
In Milne's paper \cite{Milne75}, the case $\ell=p=2$ had to be
excluded due to the state of $p$-adic cohomology at the time.  More
complete results are available, and so this restriction is no longer
needed.  This is pointed out by Milne on his web site and was implicit
in the more general results in \cite{Milne86a}.

In slightly more detail, what is needed is a bridge between
crystalline cohomology (which calculates the zeta function and
receives cycle classes) and flat cohomology (which is closely related
to the Brauer group).  This bridge is provided by the cohomology of
the de\thinspace Rham-Witt complex, generalizing Serre's Witt vector
cohomology.  Such a theory was initiated by Bloch, whose approach
required $p>2$, and this is the source of the original restriction in
\cite{Milne75}.  Soon afterward, Illusie developed a different
approach without restriction on the characteristic, and important
further developments were given by Illusie-Raynaud, Ekedahl, and
Milne.

Briefly, the theory provides groups $H^j(\XX,W\Omega^i)$ with
operators $F$ and $V$ and a spectral sequence
$$E_1^{ij}=H^j(\XX,W\Omega^i)\implies H^{i+j}_{cris}(\XX/W)$$
which degenerates at $E_1$ modulo torsion.  There are also variants
involving sheaves of cycles $ZW\Omega^i$, boundaries $BW\Omega^i$, and
logarithmic differentials $W\Omega^i_{log}$.  The various cohomology
groups (modulo torsion) are related by the spectral sequence to pieces
of crystalline cohomology defined by ``slope'' conditions, i.e., by
the valuations of eigenvalues of Frobenius.  The cohomology of the
logarithmic differentials for $i=1$ is closely related to the flat
cohomology
$$H^j(\XX,\Zp(1))=\varprojlim_n H^j(\XX,\mu_{p^n}).$$

The torsion in de\thinspace Rham-Witt cohomology can be ``large'' and
it provides further interesting invariants.

I recommend \cite{Illusie79a} and \cite{Illusie83} for a thorough
overview of the theory and \cite{Milne86a} for further important
developments.

\section{$L$-functions and the $BSD$ conjecture for $J_X$}
In this section we again assume that $k=\Fq$ and we consider $X$ and
$J_X$ satisfying hypotheses as usual.  The discussion will be parallel
to that of the preceding section.

\subsection{$L$-functions}
Choose an auxiliary prime $\ell\neq p$ and consider the $\ell$-adic
Tate module
$$V_\ell J_X=\left(\varprojlim_{n}
  J_X[\ell^n](\Kbar)\right)\tensor_\Zl\Ql$$ 
equipped with the natural action of $G_K$.  There is are canonical
isomorphisms 
$$V_\ell J_X\cong H^1(\overline{J}_X,\Ql)^*\cong
H^1(\Xbar,\Ql)^*=H^1(\Xbar,\Ql(1))$$ 
where the $*$ indicates the dual vector space.  

We define
$$L(J_X,T)=\prod_{v}\det(1-T\,\Fr_v|V_\ell^{I_v})$$
where the product runs over the places of $K$, $I_v$ denotes an
inertia group at $v$, and $\Fr_v$ denotes the {\it arithmetic\/}
Frobenius at $v$.  (Alternatively, we could use $H^1(\Xbar,\Ql)$ and
the geometric Frobenius.)  Also, we write $L(J_X,s)$ for the
$L$-function above with $T=q^{-s}$.  The product defining $L(J_X,s)$
converges in a half plane and the cohomological description below
shows that it is a rational function of $q^{-s}$ and so extends
meromorphically to all $s$.

Let $j:U\into\CC$ be a non-empty open subset over which $X$ has good
reduction.  Then the representation $G_K\to\aut(H^1(\Xbar,\Ql))$ is
unramified at all places $v\in U$ and so defines a lisse sheaf $\FF_U$
over $U$.  We set $\FF=j_*\FF_U$.  The resulting constructible sheaf
on $\CC$ is independent of the choice of $U$.  Its stalk $\FF_\xbar$
at a geometric point $\xbar$ over closed point $x\in \CC$ is the group
of inertial invariants $H^1(\Xbar,\Ql)^{I_x}$.

The Grothendieck-Lefschetz trace formula for $\FF$ reads
$$\sum_{x\in\CC(\F_{q^n})}\tr(\Fr_x|\FF_{\xbar})=
\sum_i(-1)^i\tr\left(\Fr_{q_n}|H^i(\CCbar,\FF))\right)$$
and this leads to a cohomological expression for the $L$-function as a
rational function in $T$:
\begin{equation}\label{eq:L-cohom}
L(J_X,T)=\prod_{i=0}^2Q_i(T)^{(-1)^{i+1}}
\end{equation}
where
$$Q_i(J_X,T)=\det\left(1-T\,\Fr_q|H^i(\CCbar,\FF)\right)
=\prod_{j}(1-\beta_{ij}T).$$

By Deligne's theorem, 
the $\beta_{ij}$ are Weil integers of size $q^{(i+1)/2}$.

It is convenient to have a criterion for $L(J_X,T)$ to be a polynomial.

\begin{lemma}
  $L(J_X,T)$ is a polynomial in $T$ if and only if
  $H^0(\CCbar,\FF)=H^2(\CCbar,\FF)=0$ if and only if the $K/k$-trace
  $B=\tr_{K/k}J_X$ is zero.
\end{lemma}

\begin{proof}
  The first equivalence is immediate from the cohomological
  formula~\eqref{eq:L-cohom}.  For the second equivalence, we use the
  Lang-N\'eron theorem which says that if $B$ is zero, then $J_X(\kbar
  K)$ is finitely generated.  In this case, its torsion subgroup is
  finite and so
$$H^0(\CCbar,\FF)=\varprojlim_n \left(J_X[\ell^n](\kbar K)\right)=0.$$
Conversely, if $B\neq0$, then $V_\ell B\subset V_\ell J_X$ and we see
that $H^0(\CCbar,\FF)$ contains a subspace of dimension $2\dim B>0$.
Thus $B=0$ is equivalent to $H^0(\CCbar,\FF)=0$.  That these
statements are equivalent to $H^2(\CCbar,\FF)=0$ follows from
duality (using autoduality up to a Tate twist of $\FF$).
\end{proof}

Whether or not $L(J_X,s)$ is a polynomial, its zeroes lie on the line
$\RP s=1$, and the order of zero at $s=1$ is equal to the multiplicity
of $q$ as an eigenvalue of $\Fr_q$ on $H^1(\CCbar,\FF)$.

\subsection{The basic $BSD$ conjecture}
\begin{conj}[$BSD$] We have
$$\rk J_X(K)=\ord_{s=1} L(J_X,s).$$
\end{conj}

Parallel to the Tate conjecture case, we have an inequality 
and implications among the conjectures.

\begin{prop} We have
\begin{enumerate}
\item $\rk J_X(K)\le\ord_{s=1} L(J_X,s)$
\item $BSD$ \textup{(}i.e., equality in (1)\textup{)} implies the
  Tate-Shafarevich conjecture $TS(J_X,\ell)$ \textup{(}and thus also
  the Selmer group conjecture $S(J_X,\ell)$\textup{)} for all $\ell$
\item $TS(J_X,\ell)$ for any one $\ell$ implies $BSD$.
\end{enumerate}
\end{prop}

\begin{proof}
The discussion in Section~\ref{s:Sel-H1R1} shows that
\begin{align*}
&-\ord_{s=1}\zeta(\XX,s)-\ord_{s=1}L(J_X,s)\\
=&\dim H^2(\XXbar,\Ql(1))^{G_k}-\dim H^1(\CCbar,\R^1\pi_*\Ql(1))^{G_k}\\
=&2+\sum_{v}(f_v-1).
\end{align*}
By the Shioda-Tate theorem,
$$\rk\NS(\XX)-\rk J_X(K)=2+\sum_{v}(f_v-1).$$
Thus
$$\rk J_X(K)-\ord_{s=1}L(J_X,s)=\rk\NS(\XX)+\ord_{s=1}\zeta(\XX,s)$$
and so $\rk J_X(K)\le \ord_{s=1}L(J_X(K),s)$ with equality if and only
if $T_2(\XX)$ holds.  But $T_2(\XX)$ implies the finiteness of the
$\ell$ primary part of $\Br(\XX)$ for all $\ell$ and thus finiteness
of the $\ell$ primary part of $\sha(J_X)$ for all $\ell$.  By
Proposition~\ref{prop:Sel<=>TS} his proves that $BSD$ implies the
Selmer group conjecture $S(J_X,\ell)$ for all $\ell$

Conversely, if $S(J_X,\ell)$ holds for some $\ell$, then the
$\ell$-primary part of $\sha(J_X)$ is finite and so is the $\ell$
primary part of $\Br(\XX)$.  By Proposition~\ref{prop:T1<=>Br} this
implies $T_1(\XX,\ell)$ and therefore $T_2(\XX)$. By the previous
paragraph, this implies $BSD$ as well.
\end{proof}

\subsection{The refined $BSD$ conjecture}
We need two more invariants.

Recall that we have the canonical height pairing
\begin{align*}
J_X(K)\times J_X(K)&\to\R\\
(P,Q)&\mapsto\langle P,Q\rangle
\end{align*}
which is bilinear and symmetric, and which takes values in $\Q\cdot\log
q$.  It makes $J_X(K)/tor$ into a Euclidean lattice.  We define the
regulator of $J_X$ to be the discriminant of the
height pairing:
$$R=R(J_X)=\left|\det\left(\langle P_i,P_j\rangle\right)_{i,j=1,\dots,r}\right|$$
where $P_1,\dots,P_r$ is a basis of $J_X(K)/tor$.
The regulator is $(\log q)^r$ times  a rational number (with an {\it a
  priori\/} bounded denominator). 

We also define the Tamagawa number of $J_X$ as follows: Choose an
invariant differential $\omega$ on $J_X$ over $K$.  At each place $v$,
let $a_v$ be the integer such that $\pi_v^{a_v}\omega$ is a N\'eron
differential at $v$. (Here $\pi_v$ is a generator of the maximal ideal
$\m_v$ at $v$.)  Also, let $c_v$ be the number of connected components
in the special fiber of the N\'eron model at $v$.  Then we set
$$\tau=\tau(J_X)=q^{g_X(1-g_\CC)}\prod_vq^{g_Xa_v}c_v.$$
(See \cite{Tate66b} for a less {\it ad hoc\/} version of this definition.)

The refined $BSD$ conjecture ($rBSD$ for short) relates the leading
coefficient of the $L$-function at $s=1$ to the other invariants
attached to $J_X$.

\begin{conj}[$rBSD$] $BSD$ holds, $\sha(J_X)$ is finite, and we have
  the asymptotic
$$L(J_X,s)\sim\frac{R|\sha(J_X)|\tau}{|J_X(K)_{tor}|^2}(s-1)^r$$
as $s\to1$.  
\end{conj}

As with the Tate and Artin-Tate conjectures, the basic conjecture
implies the refined conjecture.

\begin{thm}[Kato-Trihan]
  If $J_X$ satisfies the $BSD$ conjecture, then it also satisfies the
  $rBSD$ conjecture.
\end{thm}

(Kato and Trihan \cite{KatoTrihan03} actually treat the general case
of abelian varieties, not just Jacobians.)

This is a difficult theorem which completes a long line of research by
many authors, including Artin, Tate, Milne, Gordon, Schneider, and
Bauer.  Its proof is well beyond what can be covered in these
lectures, but we sketch a few of the main ideas.

In \cite{Schneider82}, Schneider generalized the results of Artin-Tate
to abelian varieties: $|\sha(A)[\ell^\infty]|<\infty$ for any one
$\ell\neq p\implies$ the prime-to-$p$ part of $rBSD$.  In
\cite{Bauer92}, Bauer was able to handle the $p$-part for abelian
varieties with everywhere good reduction.  In all of the above, the
main thing is to compare a ``geometric cohomology'' (such as
$H^i(\XXbar,\Zl)$ or crystalline cohomology) that computes the
$L$-function to an ``arithmetic cohomology'' (such as the flat
cohomology of $\mu_n$) that relates to $A(K)$ and $\sha(A)$.
Ultimately, on the geometric side one needs an integral theory which
is supple enough to handle degenerating coefficients.  Log crystalline
cohomology does the job in the context of semistable reduction.  The
Kato-Trihan paper \cite{KatoTrihan03} handles the general case by an
elaborate argument involving passing to an extension of $K$ where $A$
achieves semistable reduction, using log-syntomic cohomology upstairs
to make a comparison with flat cohomology, and showing that the
results can be brought back down to $K$ by a Galois argument.  It is a
{\it tour de force\/} of $p$-adic cohomology theory.

\section{Summary of implications}\label{s:analytic-comparison}

The results of this and the preceding two chapters show that many
statements related to the BSD and Tate conjecture are equivalent.  In
fact, there are many redundancies, so the reader is invited to choose
his or her favorite way to organize the implications.  The net result
is the following.

\begin{thm}\label{thm:main-comparison}
Let $\XX$ be a smooth, proper, geometrically irreducible surface over
a finite field $k$, equipped with a generically smooth morphism
$\XX\to\CC$ to a smooth, proper, geometrically irreducible curve
$\CC$.  Let $K$ be the function field $k(\CC)$, let $X\to\spec K$ be
the generic fiber of $\pi$, and let $J_X$ be the Jacobian of $X$.  We
have
\begin{equation}\label{eq:ineq}
\ord_{s=1}L(J_X,s)-\rk J_X(K)
=-\ord_{s=1}\zeta(\XX,s)-\rk\NS(\XX)
\ge0
\end{equation}
and the following are equivalent:
\begin{itemize}
\item Equality holds in \eqref{eq:ineq}.
\item $\sha(J_X)[\ell^\infty]$ is finite for any one prime number $\ell$.
\item $\Br(\XX)[\ell^\infty]$ is finite for any one prime number $\ell$.
\item $\sha(J_X)$ is finite.
\item $\Br(\XX)$ is finite. 
\end{itemize}
If these conditions hold, then the refined $BSD$ conjecture
$rBSD(J_X)$ and the Artin-Tate conjecture $AT(\XX)$ hold as well.
\end{thm}


To end the chapter, we quote one more precise result
\cite{LiuLorenziniRaynaud05} on the connection between $\sha(J_X)$ and
$\Br(\XX)$:

\begin{prop}[Liu-Lorenzini-Raynaud]
  Assume that the equivalent conditions of
  Theorem~\ref{thm:main-comparison} hold.  Then the order of $\Br(\XX)$ is
  a square and we have
$$|\sha(J_X)|\prod_v\delta_v\delta'_v=|\Br(\XX)|\delta^2.$$
Here $\delta$ is the index of $X$, and $\delta_v$ and $\delta'_v$ are
the index and period of $X\times_KK_v$.
\end{prop}

We refer to \cite{LiuLorenziniRaynaud05} for the interesting history
of misconceptions about when the order of the Brauer group or
Tate-Shafarevich group is a square.

\chapter{Complements}\label{ch:complements}

\section{Abelian varieties over $K$}
Assume that $k$ is finite and let $A$ be an abelian variety over $K$.
Then the $BSD$ and refined $BSD$ conjectures make sense for $A$.  (See
\cite{KatoTrihan03} for the statement---it involves also the dual
abelian variety.)  Kato-Trihan and predecessors proved in this case
too that $\rk A(K)\le\ord_{s=1}L(A,s)$ with equality if and only if
$\sha(A)[\ell^\infty]$ is finite for one $\ell$ if and only if
$\sha(A)$ is finite, and that when these conditions hold, the refined
$BSD$ conjecture does too.

We add one simple observation to this story:

\begin{lemma}
If the $BSD$ conjecture holds for Jacobians over $K$, then it also holds
for abelian varieties over $K$.
\end{lemma}

\begin{proof}
  It is well known that given an abelian variety $A$ over $K$, there
  is another abelian variety $A'$ over $K$ and a Jacobian $J$ over $K$
  with an isogeny $J\to A\oplus A'$.  If $BSD$ holds for Jacobians, then
  it also holds for $A\oplus A'$.  But since we have an inequality
  ``rank $\le$ ord'' for abelian varieties, equality for the direct
  sum implies equality for the factors.  Thus $BSD$ holds for $A$ as
  well.
\end{proof}

\section{Finite $k$, other special value formulas}
\label{s:higher dimensions}
Special value conjectures for varieties over finite fields can be made
more streamlined by stating them in terms of Euler characteristics of
cohomology of suitable complexes of sheaves (so called ``motivic
cohomology'') and the whole zeta function (as opposed to just the
piece corresponding to one part of cohomology)---contrast with the
Artin-Tate conjecture which relates to $H^2$.  See for example
\cite{Schneider82}, \cite{Lichtenbaum83}, and \cite{Milne86a}.  See
also the note on \cite{Milne86a} on Milne's web site for a
particularly streamlined statement using Weil-\'etale cohomology.

\section{Finitely generated $k$}
The analytic and algebraic conjectures discussed in the preceding
chapters have generalizations to the case where $k$ is finitely
generated, rather than finite.  In this section we give a quick
overview of the statements and a few words about some of the
comparisons.  

Assume that $k$ is finitely generated over its prime field.  We assume
as always that $\XX$ is a smooth, proper, geometrically irreducible
surface over $k$ equipped with a flat, projective, relatively minimal
morphism $\pi:\XX\to\CC$ where $\CC$ is a smooth, geometrically
irreducible curve over $k$.  As usual, $K=k(\CC)$ is the function
field of $\CC$ and $X\to\spec K$ is the generic fiber of $\pi$.

\subsection{Algebraic conjectures}
Fix a prime $\ell$ (with $\ell=p$ allowed).  The conjectures
$T_1(\XX,\ell)$ and $S(J_X,\ell)$ have straightforward generalizations
to the case where $k$ is finitely generated.  

As before, we have the exact sequence
$$0\to\NS(\XXbar)\tensor\Ql\to H^2(\XXbar,\Ql(1))\to V_\ell
\Br(\XXbar)\to 0$$
induced from the Kummer sequence as in Subsection~\ref{ss:seqs}. 

Taking $G_k$ invariants, we have an exact sequence
$$0\to\NS(\XX)\tensor\Ql\to H^2(\XXbar,\Ql(1))^{G_k}\to (V_\ell
\Br(\XXbar))^{G_K}\to0.$$ 
We have a 0 on the right since $H^1(k,\NS(\XXbar)\tensor\Ql)$ is both a
torsion group and a $\Ql$ vector space.  Conjecture $T_1(\XX,\ell)$ is
the statement that the first map above is an isomorphism.  

The exact sequence above shows that vanishing of a certain Brauer
group (namely $(V_\ell \Br(\XXbar))^{G_K}$) is equivalent to
$T_1(\XX,\ell)$.  Note that one cannot expect that $\Br(\XX)$ is
finite in general---it may contain copies of $\Br(k)$.

It is also straightforward to generalize the Selmer group and
Tate-Shafarevich conjectures.  Define groups $\Sel(J_X,\Zl)$ and
$\sha(J_X)$ exactly as in Subsection~\ref{ss:Sel-Sha}.  Then we have
an exact sequence
$$0\to J_X(K)\tensor\Zl\to\Sel(J_X,\Zl)\to T_\ell\sha(J_X)\to0.$$ 
The Selmer group conjecture is that the first map is an isomorphism,
and this is obviously equivalent to the vanishing of $T_\ell\sha(J_X)$.  

\subsection{Setup for analytic conjectures}
Recall that $k$ is assumed to be finitely generated over its prime
field.  This means there is an irreducible, regular scheme $Z$ of
finite type over $\spec\Z$ whose function field is $k$.  Of course $Z$
is only determined up to birational isomorphism and we may shrink it
in the course of the discussion.  We write $d$ for the dimension of
$Z$.

We now choose models of the data over $Z$.  
That is,
we choose schemes
$\tilde\XX$ and $\tilde\CC$ with morphisms 
\begin{equation*}
\xymatrix{\tilde\XX\ar^h[rr]\ar_f[rd]&&\tilde\CC\ar^g[ld]\\
&Z}
\end{equation*}
such that the generic fiber is 
\begin{equation*}
\xymatrix{\XX\ar^\pi[rr]\ar[rd]&&\CC\ar[ld]\\
&\spec k.}
\end{equation*}
Possibly after shrinking $Z$, we can and will assume that $f$, $g$,
and $h$ are smooth and proper.  For a closed point $z\in Z$, the
residue field at $z$ is finite, and we assume that the fiber over $z$
of the diagram above satisfies our usual hypotheses (in particular
$h_z$ should be relatively minimal).  Shrinking $Z$ further, we may
assume that $h$ is ``equisingular'' in the sense that as $\zbar$
varies through geometric points over closed points $z\in Z$, the
fibers $\tilde\XX_\zbar\labeledto{h_\zbar}\tilde\CC_\zbar$ have the
same number of singular fibers with the same configurations of
components.  (More precisely, we may assume that the set of critical
values of $h$ is \'etale over $Z$.)

\subsection{Analytic conjectures}
For the rest of this section, $\ell$ is a prime $\neq\ch(k)$.
For each smooth, projective scheme $V$ of dimension $n$ over a finite
field of cardinality $q$, we write the
zeta function of $V$ as
$$\zeta(V,s)=\frac{P_1(V,s)\cdots P_{2n-1}(V,s)}{P_0(V,s)\cdots
  P_{2n}(V,s)}$$
where the $P_i$ are the characteristic polynomials of Frobenius at
$q^{-s}$:
$$P_i(V,s)=\det\left(1-q^{-s}\Fr_q|H^i(\overline{V},\Ql)\right).$$

Now define
$$\Phi_2(\XX/k,s)=\prod_{z\in Z}P_2(\tilde\XX_z,s)^{-1}.$$
Here the product is over closed points $z$ of $Z$ and $\tilde\XX_z$ is
the fiber of $f$ over $z$, which by our hypotheses is a smooth
projective surface over the residue field at $z$.

Similarly, define
$$\Phi_1(X/K,s)=\prod_{c\in \tilde \CC} P_1(\tilde\XX_c,s)^{-1}$$
here the product is now over the closed points of $\CC$.  In this
case, the fibers are projective curves, but they will not in general
be smooth or irreducible and so we need a slight extension of our
definition 
of
the polynomials $P_1$.  (For the order of vanishing
conjecture to follow, we could shrink $\tilde\CC$ to avoid this
problem, but for the comparison between conjectures to follow, it is
more convenient to set up the definitions as we have.)

With these definitions, we have conjectures:
$$T_2(\XX):\qquad-\ord_{s=d+1}\Phi_2(\XX/k,s)=\rk\NS(\XX)$$
and
$$BSD(X):\qquad \ord_{s=d+1}\Phi_1(X/K,s)=\rk J_X(K).$$

\subsection{Comparison of conjectures}
Comparing the various conjectures for finitely generated $k$ seems to
be much more complicated than for finite $k$.

The Shioda-Tate isomorphism  \ref{prop:S-T} works for general $k$ and
gives a connection between the finitely generated groups $\NS(\XX)$,
$J_X(K)$, and $B(k)$.  

To discuss the zeta function side of the analytic conjectures, we
make the assumption that $\ch(k)=p>0$.  (In characteristic 0, very
little can be said about zeta functions without automorphic techniques
which are far outside the scope of these notes.)  So let us assume
that $Z$ has characteristic $p$ and that its field of constants has
cardinality $q$.  

In this case, the Grothendieck-Lefschetz trace
formula shows that 
the 
order of pole of $\Phi_2(\XX/k,s)$ at
$s=d+1$ is equal to the multiplicity of $q^{d+1}$ as an eigenvalue of
$\Fr_q$ on $H^{2d}_c(\overline Z,R^2f_*\Ql)$ for any $\ell\neq p$.
The Leray spectral sequence for $f=gh$ relates $R^2f_*\Ql$ to
$R^2g_*(h_*\Ql)$, $R^1g_*(R^1h_*\Ql)$, and $g_*(R^2h_*\Ql)$.  It is not
hard to work out $H^{2d}_c(\overline Z,R^2g_*(h_*\Ql))$ and
$H^{2d}_c(\overline Z,g_*(R^2h_*\Ql))$ and one finds that they
contribute exactly $1+\sum_v(f_v-1)$ (notation as in
Corollary~\ref{cor:STcor}) to the order of pole.  

It remains to consider $H^{2d}_c(\overline Z,R^1g_*(R^1h_*\Ql))$.  The
Leray spectral sequence for $h$ leads to an exact sequence 
\begin{multline*}
0\to H^{2d}_c(\overline Z,R^1g_*(R^1h_*\Ql))\to 
H^{2d+1}_c(\overline {\tilde\CC},R^1h_*\Ql)\\
\to H^{2d-1}_c(\overline Z,R^2g_*(R^1h_*\Ql))\to0
\end{multline*}
The middle group is exactly what controls $\Phi_1(X/K,s)$ (in other
words, its order of vanishing at $s=d+1$ is the multiplicity of
$q^{d+1}$ as eigenvalue of $\Fr_q$ on this group).  On the other hand,
after some unwinding, one sees that the group $H^{2d-1}_c(\overline
Z,R^2g_*(R^1h_*\Ql))$ is related to $B$ via a $BSD$ conjecture.  

Thus it seems possible (after filling in many details) to show that
$T_2(\XX)$ and $BSD(J_X)+BSD(B)$ should be equivalent in positive
characteristic.  (We remark that Tate's article \cite{Tate65} gives a
different analytic comparison, namely between $T_2(\XX)$ and a $BSD$
conjecture related to the rank of $\Picvar_\XX$, for general finitely
generated $k$.)

A comparison between the algebraic conjectures along the lines of
Proposition~\ref{prop:Sel-H1R1} looks like an interesting project.  We
remark that it will certainly be more complicated than over a finite
$k$.  For example, if $K'=K\kbar$, then in the limit
$$\varprojlim_n J_X(K')/\ell^n$$
the subgroup $B(\kbar)$ dies as it is $\ell$-divisible.  Thus the
kernel of a map $\Sel(J_X,\Zl)\to\Sel(J_X/K',\Zl)^{G_k}$ cannot be
finite in general.

Finally, a comparison of analytic and algebraic conjectures, along the
lines of the results of Artin-Tate and Milne (as always assuming $k$
has positive characteristic), also seems to be within the realm of
current technology (although the author does not pretend to have
worked out the details).

I know of no strong evidence in favor of the conjectures in this
section beyond the case where $k$ is a global field.

\section{Large fields $k$}
We simply note that if $k$ is ``large'' then one cannot expect
finiteness of $\Br(\XX)$ or $\sha(J_X)$.  For example, the exact
sequence
$$0\to\NS(\XXbar)\tensor\Zl\to
H^2(\XXbar,\Zl(1))\to T_\ell\Br(\XXbar)\to 0$$ 
shows that over a separably closed field, when the rank of $H^2$ is
larger than the Picard number, $\Br$ has divisible elements.

It may also happen that $\Br(\XX)$ or $\sha(J_X)$ has infinite
$p$-torsion.  For examples in the case of $\sha$ (already for elliptic
curves), see \cite[7.12b]{Ulmer91}, where there is a
Selmer group which is in a suitable sense linear as a function of the
finite ground field, and is thus infinite when $k=\Fpbar$.  This
phenomenon is closely related to torsion in the de\thinspace Rham-Witt
cohomology groups.  See \cite{Ulmer96a} for some related issues.

\chapter{Known cases of the Tate conjecture and
  consequences}\label{ch:cases} 

\section{Homomorphisms of abelian varieties 
and $T_1$ for products}
Let $k$ be a field finitely generated over its prime field.  In the
article where he first conjectured $T_1(\XX,\ell)$ for smooth
projective varieties $\XX$ over $k$, Tate explained that the case of
abelian varieties is equivalent to an attractive statement
on homomorphisms of abelian varieties.  Namely, if $A$ and $B$ are
abelian varieties over $k$, then for any $\ell\neq\ch(k)$, the
natural homomorphism
\begin{equation}\label{eq:homs}
\Hom_k(A,B)\tensor\Zl\to \Hom_{G_k}(T_\ell A,T_\ell B)
\end{equation}
should be an isomorphism.

In \cite{Tate66a}, Tate gave an axiomatic framework showing that
isomorphism in \eqref{eq:homs} follows from a fundamental finiteness
statement for abelian varieties over $k$, which is essentially trivial 
for finite $k$.  The finiteness statement (and thus isomorphism in
\eqref{eq:homs}) was later proven for all finitely generated
fields of characteristic $p>0$ by Zarhin \cite{Zarhin76} ($p>2$) and
Mori ($p=2$); see \cite{MoretBailly85}, 
and for finitely generated
fields of characteristic zero by Faltings \cite{Faltings86}. 

Isomorphism in \eqref{eq:homs} in turn implies $T_1(\XX,\ell)$ for
products of curves and abelian varieties (and more generally for
products of varieties for which $T_1$ is known).  This is explained,
for example, in \cite{Tate66a} or \cite{Ulmer11}.

Summarizing the part of this most relevant for us, we have the
following statement.

\begin{thm}\label{thm:T1products}
If $k$ is a finitely generated field and $\XX$ is a product of curves
over $k$, then for all $\ell\neq\ch(k)$ we have an isomorphism
$$\NS(\XX)\tensor\Zl\isoto H^2(\XXbar,\Zl(1))^{G_k}$$
\end{thm}


\section{Descent property of $T_1$ and domination by a 
product of curves}
The
property in question is the following.

\begin{lemma}\label{lemma:domination}
Suppose $k$ is a finitely generated field and $\XX$ is smooth, proper
variety over $k$ satisfying $T_1$.  If $\XX\ratto\YY$ is a dominant
rational map, then $T_1(\YY)$ holds as well.
\end{lemma}

This is explained in \cite{Tate94} and, with an unnecessary hypothesis
on resolution of singularities, in \cite{Schoen96}.  The latter article
points out the utility of combining Theorem~\ref{thm:T1products} and
Lemma~\ref{lemma:domination} to prove the following result.

\begin{cor}\label{cor:dpc}
If $k$ is a finitely generated field and $\XX$ is a variety admitting a
dominant rational map from a product of curves
$\prod_i\CC_i\ratto\XX$, then $T_1$ holds for $\XX$.
\end{cor}

The corollary implies most of the known cases of $T_1$ for surfaces
(some of which were originally proven by other methods).  Namely, we
have $T_1$ for: rational surfaces, unirational surfaces, Fermat
surfaces, and more generally hypersurfaces of dimension 2 defined by
an equation in 4 monomials (with mild hypotheses).  We refer to
\cite{Shioda86} for these last surfaces, which Shioda calls ``Delsarte
surfaces.'' See also \cite{Ulmer02} and \cite{Ulmer07b}.

Of course, for finite $k$, surfaces which satisfy $T_1$ lead to curves
whose Jacobians satisfy $BSD$.  We will explain in the next chapter how
this leads to Jacobians of large rank over global function fields.

\begin{rem}
  In a letter to Grothendieck dated 3/31/64
  \cite{GrothendieckSerre01}, Serre constructs an example of a surface
  which is not dominated by a product of curves.  In a note to this
  letter, Serre says that Grothendieck had hoped to prove the Weil
  conjectures by showing that every variety is dominated by a product of
  curves, thus reducing the problem to the known case of curves.
  Thanks to Bruno Kahn for pointing out this letter.  See also
  \cite{Schoen96} for other examples of varieties not dominated by
  products of curves.
\end{rem}

\section{Other known cases of $T_1$ for surfaces}
Assume that $k$ is a finite field.  The other main systematic cases of
$T_1$ for surfaces over $k$ are for $K3$ surfaces.  Namely, Artin and
Swinnerton-Dyer showed in \cite{ArtinSwinnertonDyer73} that $\sha(E)$
is finite for an elliptic curve $E$ over $k(t)$ when the corresponding
elliptic surface $\EE\to\P^1_k$ is a K3 surface.  And Nygaard and Ogus
showed in \cite{NygaardOgus85} that if $\XX$ is a $K3$ surface of
finite height and $\ch(k)\ge5$, then $T_1$ holds for $\XX$.

In a recent preprint, Lieblich and Maulik show that the Tate
conjecture hold for $K3$ surfaces over finite fields of characteristic
$p\ge 5$ if and only if there are only finitely many isomorphism
classes of $K3$s over each finite field of characteristic $p$.  This
is reminiscent of Tate's axiomatization of $T_1$ for abelian
varieties.

It was conjectured by Artin in \cite{Artin74} that a K3 surface of
infinite height has N\'eron-Severi group of rank 22, the maximum
possible, and so this conjecture together with \cite{NygaardOgus85}
would imply the Tate conjecture for $K3$ surfaces over field of
characteristic $\ge5$.  Just as we are finishing these notes, Maulik
and Madapusi Pera have announced (independently) results that would
lead to a proof of the Tate conjecture for $K3$s with a polarization
of low degree or of degree prime to $p=\ch(k)$.

Finally, we note the ``direct'' approach:  We have {\it a priori\/}
inequalities 
$$\rk\NS(\XX)\le\dim_\Ql H^2(\XX,\Ql(1))^{G_k}\le-\ord_{s=1}\zeta(\XX,s)$$
and if equality holds between the ends or between the first two terms,
then we have the full Tate conjectures.  It is sometimes possible to
find, say by geometric construction, enough cycles to force equality.
We will give an example of this in the context of elliptic surfaces in
the next chapter.

\chapter{Ranks of Jacobians}\label{ch:ranks}

In this final chapter, we discuss how the preceding results on the
$BSD$ and Tate conjectures can be used to find examples of Jacobians
with large rank analytic and algebraic rank over function fields over
finite fields.  In the first four sections we briefly review results
from recent publications, and in the last three sections we state a
few results that will appear in future publications.

\section{Analytic ranks in the Kummer tower}
In this section $k$ is a finite field, $K=k(t)$, and for each
positive integer $d$ prime to $p$, $K_d=k(t^{1/d})$.  

It turns out that roughly half of all abelian varieties defined over
$K$ have unbounded analytic rank in the tower $K_d$.  More precisely,
those that satisfy a simple parity condition have unbounded rank.

\begin{thm}\label{thm:analytic-ranks}
  Let $A$ be an abelian variety over $K$ with Artin conductor
  $\n$ (an effective divisor on $\P^1_k$).  Write $\n'$ for the part of
  $\n$ prime to the places 0 and $\infty$ of $K$.  Let $\Swan_v(A)$ be
  the exponent of the Swan conductor of $A$ at a place $v$ of $K$.
  Suppose that the integer
$$\deg(\n')+\Swan_0(A)+\Swan_\infty(A)$$
is odd.  Then there is a constant $c$ depending only on $A$ such that if
$d=p^f+1$ then
$$\ord_{s=1}L(A/K_d,s)\ge\frac{p^f+1}{o_d(q)}-c$$
where $o_q(d)$ is the order of $q$ in $(\Z/d\Z)^\times$.
\end{thm}

This is a slight variant of \cite[Theorem~4.7]{Ulmer07b} applied to
the representation of $G_K$ on $V_\ell A$.  Note that $o_d(q)\le2f$ so
the rank tends to infinity with $f$.

\section{Jacobians of large rank}
Theorem~\ref{thm:analytic-ranks} gives an abundant supply of abelian
varieties with large analytic rank.  Using Corollary~\ref{cor:dpc}, we
can use this to produce Jacobians of large algebraic rank.

\begin{thm}\label{thm:Jac-ranks}
Let $p$ be an odd prime number, $K=\Fp(t)$, and $K_d=\Fp(t^{1/d})$.
Choose a positive integer $g$ such that $p\nodiv(2g+2)(2g+1)$ and let
$X$ be the hyperelliptic curve over $K$ defined by
$$y^2=x^{2g+2}+x^{2g+1}+t.$$
Let $J_X$ be the Jacobian of $X$, an abelian variety of dimension $g$
over $K$.  Then for all $d$ the $BSD$ conjecture holds for $J_X$ over
$K_d$, and there is a constant depending only on $p$ and $g$ such that
for all $d$ of the form $d=p^f+1$ we have
$$\rk J_X(K_d)\ge\frac{p^f+1}{2f}-c.$$
\end{thm}

This is one of the main results of \cite{Ulmer07b}.  We sketch the key
steps in the proof.  The analytic rank of $J_X$ is computed using
Theorem~\ref{thm:analytic-ranks}.  Because $X$ is defined by an
equation involving 4 monomials in 3 variables, by Shioda's
construction it turns out that the surface $\XX\to\P^1$ associated to
$X/K_d$ is dominated by a product of curves.  Therefore $T_2$ holds
for $\XX$ and $BSD$ holds for $J_X$.  In \cite{Ulmer07b} we also
checked that $J_X$ is absolutely simple and has trivial $K/\Fq$
trace for all $d$.

In \cite{Ulmer07b} we gave other examples so that for every $p$ and
every $g>0$, there is an absolutely simple Jacobian of dimension $g$
over $\Fp(t)$ which satisfies $BSD$ and has arbitrarily large rank in
the tower of fields $K_d$.

\section{Berger's construction}
The results of the preceding sections how that Shioda's 4-monomial
construction is a powerful tool for deducing $BSD$ for certain specific
Jacobians.  However, it has the defect that it is rigid---the property
of being dominated by a Fermat surface, or more generally by a product
of Fermat curves, does not deform.

In her thesis, Berger gave a more flexible construction which leads to
{\it families\/} of Jacobians satisfying $BSD$ in each layer of the
Kummer tower.  We quickly sketch the main idea; see \cite{Berger08}
and \cite{UlmerDPCT} for more details.

Let $k$ be a general field and fix two smooth, proper, geometrically
irreducible curves $\CC$ and $\DD$ over
$k$.  Fix also two separable, non-constant rational functions
$f\in k(\CC)^\times$ and $g\in k(\DD)^\times$ and form the rational
map
$$\CC\times\DD\ratto\P^1\qquad(x,y)\mapsto f(x)/g(y).$$
Under mild hypotheses on $f$ and $g$, the generic fiber of this map
has a smooth, proper model $X\to\spec K=k(t)$.

Berger's construction is designed so that the following is true.

\begin{thm}[Berger] With notation as above, for all $d$ prime to
  $p=\ch(k)$, the surface $\XX_d\to\P^1$ associated to $X$ over
  $K_d=k(t^{1/d})\cong k(u)$ is dominated by a product of curves.
  Thus if $k$ is finite, the $BSD$ conjecture holds for $J_X$ over $K_d$.
\end{thm}

It is convenient to think of the data in Berger's construction as
consisting of a discrete part, namely the genera of $\CC$ and $\DD$
and the combinatorial type of the divisors of $f$ and $g$, and a
continuous part, namely the moduli of the curves and the locations of
the zeros and poles of the functions.  From this point of view it is
clear that there is enormous flexibility in the choice of data.  In
\cite{Berger08}, Berger used this flexibility to produce families of
elliptic curves over $\Fp(t)$ such that for almost all specializations
of the parameters to values in $\Fq$, the resulting elliptic curve
over $\Fq(t)$ satisfies $BSD$ and has unbounded rank in the tower
$K_d=\Fq(t^{1/d})$.

Note that the values of $d$ which give high ranks (via
Theorem~\ref{thm:analytic-ranks}) are those of the form $p^f+1$ and
there is a well-known connection between such values and
supersingularity of abelian varieties over extensions of $\Fp$.  Using
Berger's construction and ideas related to those in the next section,
Occhipinti produced elliptic curves over $\Fp(t)$ which have high rank
at {\it every\/} layer of the tower $\Fpbar(t^{1/d})$.  More
precisely, he finds curves $E/\Fp(t)$ such that $\rk
E(\Fpbar(t^{1/d})\ge d$ for all $d$ prime to $p$.  This seems to be a
completely new phenomenon.  See \cite{OcchipintiThesis} and the paper
based on it for details.

\section{Geometry of Berger's construction}
In \cite{UlmerDPCT}, we made a study of the geometry of Berger's
construction with a view toward more explicit results on ranks and
Mordell-Weil groups.  

To state one of the main results, let $\CC$, $\DD$, $f$, and $g$ be as
in Berger's construction.  Define covers $\CC_d\to\CC$ and
$\DD_d\to\DD$ by the equations
$$z^d=f(x)\qquad w^d=g(x).$$
We have an action of $\mu_d$ on $\CC_d$ and $\DD_d$ and their
Jacobians.  It turns out that the surface $\XX_d\to\P^1$ associated to
$X$ over $K_d=k(t^{1/d})$ is birational to the quotient of
$\CC_d\times\DD_d$ by the diagonal action of $\mu_d$.  This explicit
domination by a product of curves makes it possible to compute the
rank of $J_X$ over $K_d$ in terms of invariants of $\CC_d$ and
$\DD_d$.  More precisely,

\begin{thm}\label{thm:DPCT}
  Suppose that $k$ is algebraically closed.  Then there exist positive
  integers $c_1$, $c_2$, and $N$ such that for all $d$ relatively
  prime to $N$ we have
$$\rk \MW( J_X/K_d)=\rk\Hom_{k-av}(J_{\CC_d},J_{\DD_d})^{\mu_d}-dc_1+c_2.$$
Here the superscript indicates those homomorphisms which commute
with the action of $\mu_d$ on the two Jacobians.
\end{thm}

The constants $c_1$ and $c_2$ are given explicitly in terms of the
input data, and the integer $N$ is there just to make the statement
simple---there is a formula for the rank for any value of $d$.  See
\cite[Thm.~6.4]{UlmerDPCT} for the details.

The theorem relates ranks of abelian varieties over a function field
$K_d$ to homomorphisms of abelian varieties over the constant field
$k$.  The latter is sometimes more tractable.  For example, when
$k=\Fpbar$, the homomorphism groups can be made explicit via
Honda-Tate theory.

Here is another example where the endomorphism side of the formula in
Theorem~\ref{thm:DPCT} is tractable: Using it and Zarhin's results on
endomorphism rings of superelliptic Jacobians, we gave examples in
\cite{UlmerZarhin10} of Jacobians over function fields of
characteristic zero with bounded ranks in certain towers.  More
precisely:

\begin{thm}\label{thm:UZ}
Let $g_X$ be an integer $\ge2$ and let $X$ be the smooth, proper curve
of genus $g_X$ over $\Q(t)$ with affine plane model
$$ty^2=x^{2g_X+1}-x+t-1.$$
Then for every prime number $p$ and every integer $n\ge0$, we have
$$\rk J_X(\Qbar(t^{1/p^n}))=2g_X.$$
\end{thm}

Another dividend of the geometric analysis is that it can sometimes
be used to find explicit points in high rank situations.  (In principle
this is also true in the context of Shioda's 4-monomial construction,
but I know of few cases where it has been successfully carried out, a
notable exception being \cite{Shioda91} which relates to an isotrivial
elliptic curve.)

The following example is \cite[Thm.~8.1]{UlmerDPCT}.  The point which
appear here are related to the graphs of Frobenius endomorphisms of
the curves $\CC_d$ and $\DD_d$ for a suitable choice of data in
Berger's construction.

\begin{thm}\label{thm:dpct-explicit}
  Fix an odd prime number $p$, let $k=\Fpbar$, and let $K=k(t)$.  Let $X$
  be the elliptic curve over $K$ defined by
$$y^2+xy+ty=x^3+tx^2.$$
For each $d$ prime to $p$, let $K_d=k(t^{1/d})\cong k(u)$.  We write
$\zeta_d$ for a fixed primitive $d$-th root of unity in $k$.
Let  $d=p^n+1$, let $q=p^n$, and let
$$P(u)=\left(\frac{u^q(u^q-u)}{(1+4u)^q},
  \frac{u^{2q}(1+2u+2u^q)}{2(1+4u)^{(3q-1)/2}}
  -\frac{u^{2q}}{2(1+4u)^{q-1}}\right).$$ 
Then the points $P_i=P(\zeta_d^iu)$ for $i=0,\dots,d-1$ lie in
$X(K_d)$ and they generate a finite index subgroup of $X(K_d)$, which
has rank $d-2$.  The relations among them are that
$\sum_{i=0}^{d-1}P_i$ and $\sum_{i=0}^{d-1}(-1)^iP_i$ are torsion.
\end{thm}

The main result of \cite{UlmerDPCT} also has some implications for
elliptic curves over $\C(t)$.  We refer the reader to the last section
of that paper for details.

\section{Artin-Schreier analogues}
The results of the preceding four sections are all related to the
arithmetic of Jacobians in the Kummer tower $K_d=k(t^{1/d})$.  It
turns out that all of the results---high analytic ranks, a
Berger-style construction of Jacobians satisfying $BSD$, a rank
formula, and explicit points---have analogues for the Artin-Schreier
tower, that is for extensions $k(u)/k(t)$ where $u^q-u=f(t)$.  See a
forthcoming paper with Rachel Pries \cite{PriesUlmerAS} for more
details.

\section{Explicit points on the Legendre curve}
As we mentioned before, if one can write down enough points to fill
out a finite index subgroup of the Mordell-Weil group of an abelian
variety, then the full $BSD$ conjecture follows.  This plays out in a
very satisfying way for the Legendre curve.

More precisely, let $p$ be an odd prime, let $K=\Fp(t)$, and let
$$K_d=\Fp(\mu_d)(t^{1/d})\cong\Fp(\mu_d)(u).$$  
Consider the Legendre curve
$$E:\qquad y^2=x(x+1)(x+t)$$
over $K$.  (The signs are not the traditional ones, and $E$ is a twist
of the usual Legendre curve.  It turns out that the + signs are more
convenient, somewhat analogous to the situation with signs of Gauss
sums.) 

If we take $d$ of the form $d=p^f+1$, then there is an obvious point
on $E(K_d)$, namely $P(u)=(u,u(u+1)^{d/2})$.  Translating this point
by $\gal(K_d/K)$ yields more points:  $P_i=P(\zeta_d^iu)$ where
$\zeta_d\in\Fp(\mu_d)$ is a primitive $d$-th root of unity and
$i=0,\dots,d-1$.

In \cite{UlmerLegendre} we give an elementary proof of the following.

\begin{thm}\label{thm:LP}
The points $P_i$ generate a subgroup of $E(K_d)$ of finite index and
rank $d-2$.  The relations among them are that $\sum_{i=0}^{d-1}P_i$
and $\sum_{i=0}^{d-1}(-1)^iP_i$ are torsion.
\end{thm}

It is easy to see that the $L$-function of $E/K_d$ has order of zero at
$s=1$ {\it a priori} bounded by $d-2$, so we have the equality
$\ord_{s=1}L(E/K_d,s)=\rk E(K_d)$, i.e., the $BSD$ conjecture.  Thus we
also have the refined $BSD$ conjecture.   Examining this leads to a
a beautiful ``analytic class number formula'':  Let $V_d$ be the
subgroup of $E(K_d)$ generated by the $P_i$ and  $E(K_d)_{tor}$.  (The
latter is easily seen to be of order 8.)  Then we have
$$[E(K_d):V_d]^2=|\sha(E/K_d)|$$
and these numbers are powers of $p$.

It then becomes a very interesting question to make the displayed
quantity more explicit.  It turns out that $E$ is closely related to
the $X$ of the previous section, and this brings the geometry of
``domination by a product of curves'' into play.  Using that, we are
able to give a complete description of $\sha(E/K_d)$ and $E(K_d)/V_d$
as modules over the group ring $\Zp[\gal(K_d/K)]$.  
These results will appear in a paper in preparation.

It also turns out that we can control the rank of $E(K_d)$ for general
$d$, not just those of the form $p^f+1$.  Surprisingly, we find
that there is high rank for many other values of $d$, and in a precise
quantitative sense, among the values of $d$ where $E(K_d)$ has large
rank, those prime to all $p^f+1$ are more numerous than those not
prime to $p^f+1$.  The results of this paragraph 
will appear in publications of the author and
 collaborators, including \cite{PomeranceUlmer} and works cited there.

\section{Characteristic 0}
Despite significant effort, I have not been able to exploit Berger's
construction to produce elliptic curves (or Jacobians of any fixed
dimension) over $\C(t)$ with unbounded rank.  Roughly speaking,
efforts to increase the $\Hom(\dots)^{\mu_d}$ term, say by forcing
symmetry, seem also to increase the value of $c_1$.  (Although one
benefit of this effort came from examination of an example that led to
the explicit points in Theorem~\ref{thm:dpct-explicit} above.)

At the workshop I explained a heuristic which suggests that ranks
might be bounded in a Kummer tower in characteristic zero.
Theorem~\ref{thm:UZ} above is an example in this direction.  I do not
know how to prove anything that strong, but we do have the following
negative result.

Recall that an elliptic surface $\pi:\EE\to\CC$ has a height, which
can be defined as the degree of the invertible sheaf
$(R^1\pi_*\O_{\EE})^{-1}$.   When $\CC=\P^1$, the height is $\ge3$ if
and only if the Kodaira dimension of $\EE$ is 1.  Recall also
that there is a reasonable moduli space for elliptic surfaces of height
$d$ over $\C$ and it has dimension $10d-2$.  In the following ``very
general'' means ``belongs to the complement of countably many divisors
in the moduli space.''

\begin{thm}\label{thm:no-rank}
Suppose that $E/\C(t)$ is the generic fiber of $\EE\to\P^1$ over $\C$
where $\EE$ is a very general elliptic surface of height $d\ge3$.
Then for every finite extension $L/\C(t)$ where $L$ is rational field
\textup{(}i.e., $L\cong\C(u)$\textup{)}, we have $E(L)=0$.
\end{thm}

The theorem is proven by controlling the collection of rational curves
on $\EE$ and is a significant strengthening of a Noether-Lefschetz
type result (cf.~\cite{Cox90}) according to which the N\'eron-Severi
group of a very general elliptic surface is generated by the zero
section and a fiber.  More details on Theorem~\ref{thm:no-rank} and
the heuristic that suggested it will appear in a paper currently in
preparation.

\bibliography{database}{}
\bibliographystyle{plain}

\end{document}

%% file: formatting.tex
\numberwithin{subsection}{section}

\numberwithin{equation}{section}

\theoremstyle{plain}
\newtheorem{thm}[equation]{Theorem}
\newtheorem{prop}[equation]{Proposition}
\newtheorem{cor}[equation]{Corollary}
\newtheorem{lemma}[equation]{Lemma}
\newtheorem{conj}[equation]{Conjecture}

\theoremstyle{definition}

\theoremstyle{remark}
\newtheorem{rem}[equation]{Remark}

\newtheorem{exer}[equation]{Exercise}

\newtheorem{rem-exer}[equation]{Remark/Exercise}
\newtheorem{rem-exers}[equation]{Remark/Exercises}

%% file: macros.tex
\usepackage[OT2,T1]{fontenc}
\DeclareSymbolFont{cyrletters}{OT2}{wncyr}{m}{n}
\DeclareMathSymbol{\sha}{\mathalpha}{cyrletters}{"58}

\newcommand{\CC}{\mathcal{C}}
\newcommand{\CCbar}{\overline{\mathcal{C}}}
\newcommand{\DD}{\mathcal{D}}
\newcommand{\DDbar}{\overline{\mathcal{D}}}
\newcommand{\JJ}{\mathcal{J}}
\newcommand{\KK}{\mathcal{K}}
\renewcommand{\SS}{\mathcal{S}}
\newcommand{\TT}{\mathcal{T}}
\newcommand{\UU}{\mathcal{U}}
\newcommand{\VV}{\mathcal{V}}
\newcommand{\XX}{\mathcal{X}}
\newcommand{\XXbar}{\overline{\mathcal{X}}}
\newcommand{\YY}{\mathcal{Y}}

\newcommand{\GG}{{\mathcal{G}}}

\renewcommand{\O}{\mathcal{O}}
\newcommand{\FF}{\mathcal{F}}
\newcommand{\EE}{\mathcal{E}}

\newcommand{\OO}{\mathcal{O}}

\newcommand{\F}{\mathbb{F}}
\newcommand{\Fp}{{\mathbb{F}_p}}

\newcommand{\Fq}{{\mathbb{F}_q}}

\newcommand{\Fpbar}{{\overline{\mathbb{F}}_p}}

\newcommand{\Qbar}{{\overline{\mathbb{Q}}}}
\newcommand{\kbar}{{\overline{k}}}
\newcommand{\Kbar}{{\overline{K}}}
\newcommand{\ratto}{{\dashrightarrow}}
\newcommand{\Ql}{{\mathbb{Q}_\ell}}
\newcommand{\Zl}{{\mathbb{Z}_\ell}}

\newcommand{\Qp}{{\mathbb{Q}_p}}
\newcommand{\Zp}{{\mathbb{Z}_p}}

\newcommand{\Z}{\mathbb{Z}}
\newcommand{\Zhat}{\hat{\mathbb{Z}}}
\newcommand{\Q}{\mathbb{Q}}
\newcommand{\R}{\mathbb{R}}
\newcommand{\C}{\mathbb{C}}
\newcommand{\A}{\mathbb{A}}
\renewcommand{\P}{\mathbb{P}}
\newcommand{\Ksep}{K^{sep}}
\newcommand{\ksep}{k^{sep}}

\newcommand{\xbar}{{\overline{x}}}
\newcommand{\vbar}{{\overline{v}}}

\newcommand{\zbar}{{\overline{z}}}
\newcommand{\Xbar}{\overline{X}}
\newcommand{\etabar}{\overline{\eta}}

\renewcommand{\d}{\mathfrak{d}}
\newcommand{\m}{\mathfrak{m}}
\newcommand{\n}{\mathfrak{n}}

\newcommand{\T}{\mathcal{T}}

\newcommand{\into}{\hookrightarrow}

\newcommand{\isoto}{\tilde{\to}}
\newcommand{\tensor}{\otimes}

\newcommand{\nodiv}{\not|}

\newcommand{\G}{\mathbb{G}}

\newcommand{\labeledto}[1]{\overset{#1}{\to}}

\DeclareMathOperator{\im}{Im}
\DeclareMathOperator{\coker}{coker}
\DeclareMathOperator{\tr}{Tr}

\DeclareMathOperator{\res}{Res}

\DeclareMathOperator{\ord}{ord}
\DeclareMathOperator{\rk}{Rank}

\DeclareMathOperator{\RP}{Re}

\DeclareMathOperator{\Hom}{Hom}
\DeclareMathOperator{\aut}{Aut}
\DeclareMathOperator{\Pic}{Pic}

\DeclareMathOperator{\Picf}{\underline{Pic}}
\DeclareMathOperator{\Picvar}{PicVar}
\DeclareMathOperator{\NS}{NS}
\DeclareMathOperator{\MW}{MW}

\DeclareMathOperator{\Br}{Br}

\DeclareMathOperator{\gal}{Gal}
\DeclareMathOperator{\spec}{Spec}
\DeclareMathOperator{\proj}{Proj}

\DeclareMathOperator{\mor}{Mor}

\DeclareMathOperator{\divr}{Div}

\DeclareMathOperator{\Swan}{Swan}
\DeclareMathOperator{\Sel}{Sel}
\DeclareMathOperator{\Fr}{Fr}
\DeclareMathOperator{\ch}{Char}

\def\clap#1{\hbox to 0pt{\hss#1\hss}}